\documentclass[final,leqno]{siamltex}
\usepackage[leqno]{amsmath}
\usepackage{graphicx}% Figure
\usepackage{graphics}
\usepackage{epstopdf}
\usepackage{subfigure}% Subfigure
\usepackage{threeparttable}

\usepackage{bm}
\usepackage{amssymb}
\usepackage{leftidx}
\usepackage{empheq}
\usepackage{color}
\usepackage{makecell}
\usepackage{enumerate}
\allowdisplaybreaks
\renewcommand{\theequation}{\arabic{section}.\arabic{equation}}
\numberwithin{equation}{section}
\newtheorem{remark}{Remark}[section]
\title{Highly efficient exponential scalar auxiliary variable approaches with relaxation (RE-SAV) for gradient flows.
        \thanks{
We would like to acknowledge the assistance of volunteers in putting together this example manuscript and supplement. This work is supported by National Natural Science Foundation of China (Grant Nos: 12001336, 11901489, 12131014).}}

      \author{Zhengguang Liu
             \thanks{School of Mathematics and Statistics, Shandong Normal University, Jinan, China. Email: liuzhg@sdnu.edu.cn.}
                                       \and
             Xiaoli Li\textsuperscript{*}
             \thanks{Corresponding author: Shandong University, Jinan, Shandong, 250100, China. Email: xiaomath@sdu.edu.cn}. }
\begin{document}
\UseRawInputEncoding
\maketitle

\begin{abstract}
For the past few years, scalar auxiliary variable (SAV) and SAV-type approaches became very hot and efficient methods to simulate various gradient flows. Inspired by the new SAV approach in \cite{huang2020highly}, we propose a novel technique to construct a new exponential scalar auxiliary variable (E-SAV) approach to construct high-order numerical energy stable schemes for gradient flows. To improve its accuracy and consistency noticeably, we propose an E-SAV approach with relaxation, which we named the relaxed E-SAV (RE-SAV) method for gradient flows. The RE-SAV approach preserves all the advantages of the traditional SAV approach. In addition, we do not need any the bounded-from-below assumptions for the free energy potential or nonlinear term. Besides, the first-order, second-order and higher-order unconditionally energy stable time-stepping schemes are easy to construct. Several numerical examples are provided to demonstrate the improved efficiency and accuracy of the proposed method.
\end{abstract}

\begin{keywords}
Scalar auxiliary variable, Gradient flows, High-order scheme, RE-SAV approach, Energy stable.
\end{keywords}

    \begin{AMS}
         65M12; 35K20; 35K35; 35K55; 65Z05.
    \end{AMS}

\pagestyle{myheadings}
\thispagestyle{plain}
\markboth{ZHENGGUANG LIU AND XIAOLI LI} {RE-SAV APPROACH FOR GRADIENT FLOW}
   %==================================================================
  \section{Introduction}
The gradient flow models are very important and popular dissipative systems which cover a lot of fields such as alloy casting, new material preparation, image processing, finance and so on \cite{ambati2015review,guo2015thermodynamically,liu2019efficient,liu2020two,marth2016margination,miehe2010phase,shen2015efficient,wheeler1993computation}. Many classical gradient flow models such as Allen-Cahn model \cite{ainsworth2017analysis,du2019maximum,guan2014second,shen2010numerical,yang2020convergence,zhai2014numerical}, Cahn-Hilliard model \cite{chen2019fast,du2018stabilized,he2007large,shen2010numerical,weng2017fourier,yang2018numerical,zhu1999coarsening} and phase field crystal model \cite{li2018unconditionally,li2019efficient,li2017efficient,liu2020two,yang2017linearly} have been widely used to solve a series of physical problems. Gradient flow models are generally derived from the functional variation of free energy. In general, the free energy $E(\phi)$ contains the sum of an integral phase of a nonlinear functional and a quadratic term:
\begin{equation}\label{intro-e1}
E(\phi)=\frac12(\phi,\mathcal{L}\phi)+E_1(\phi)=\frac12(\phi,\mathcal{L}\phi)+\int_\Omega F(\phi)d\textbf{x},
\end{equation}
where $\mathcal{L}$ is a symmetric non-negative linear operator, and $E_1(\phi)=\int_\Omega F(\phi)d\textbf{x}$ is nonlinear free energy. $F(\textbf{x})$ is the energy density function. The gradient flow from the energetic variation of the above energy functional $E(\phi)$ in \eqref{intro-e1} can be obtained as follows:
\begin{equation}\label{intro-e2}
\displaystyle\frac{\partial \phi}{\partial t}=-\mathcal{G}\mu,\quad\mu=\displaystyle\mathcal{L}\phi+F'(\phi),
\end{equation}
where $\mu=\frac{\delta E}{\delta \phi}$ is the chemical potential. $\mathcal{G}$ is a positive operator. For example, $\mathcal{G}=I$ for the $L^2$ gradient flow and $\mathcal{G}=-\Delta$ for the $H^{-1}$ gradient flow.

It is not difficult to find that the above phase field system satisfies the following energy dissipation law:
\begin{equation*}
\frac{d}{dt}E=(\frac{\delta E}{\delta \phi},\frac{\partial\phi}{\partial t})=-(\mathcal{G}\mu,\mu)\leq0,
\end{equation*}
which is a very important property for gradient flows in physics and mathematics. From a mathematical point of view, whether the numerical methods can maintain the discrete energy dissipation law is an important stability indicator. Up to now, many scholars considered a series of efficient and popular time discretized approaches to construct energy stable schemes for different phase field models such as convex splitting approach \cite{eyre1998unconditionally,shen2012second,shin2016first}, linear stabilized approach \cite{shen2010numerical,yang2017numerical}, exponential time differencing (ETD) approach \cite{du2019maximum,WangEfficient}, invariant energy quadratization (IEQ) approach \cite{chen2019efficient,chen2019fast,yang2016linear}, scalar auxiliary variable (SAV) approach \cite{xiaoli2019energy,shen2018scalar,ShenA} and so on.

Compared with other numerical methods, the SAV approach has many incomparable advantages. One is that it is very easy to construct linear, second-order and unconditionally energy stable schemes. Until now, it has been applied successfully to simulate many classical gradient flows such as Allen-Cahn models \cite{yang2020convergence,shen2018scalar}, Cahn-Hilliard models \cite{yang2018numerical,li2019energy}, phase field crystal models \cite{liu2019efficient,li2020stability}, molecular beam epitaxial growth model \cite{YangNumerical,cheng2019highly}, Cahn-Hilliard-Navier-Stokes models \cite{li2019sav} and so on. It is worth mentioning that the dissipative system without gradient flow structure, such as Navier-Stokes models \cite{li2020new,lin2019numerical,xiaoli2020error} can also be simulated effectively by the SAV approach. Another advantage of SAV method is that high-order energy stable numerical schemes can be constructed successfully by Runge-Kutta method \cite{akrivis2019energy} or $k$-step backward differentiation formula (BDF$k$) \cite{huang2020highly}.

Recently, many variants of SAV schemes are developed to modify the shortcomings of the traditional SAV approach. Some of the SAV-type methods change the definition of the introduced SAV. For example, in \cite{yang2020roadmap}, the authors introduced the generalized auxiliary variable method for devising energy stable schemes for general dissipative systems. An exponential SAV approach in \cite{liu2020exponential} is developed to modify the traditional method to construct energy stable schemes by introducing an exponential SAV. A series of generalized SAV approaches which extend the applicability of the original SAV approach for gradient systems can be found in \cite{cheng2021generalized}. In \cite{huang2020highly}, the authors consider a new SAV approach to construct high-order energy stable schemes. In \cite{jiang2022improving}, Jiang et al. present a relaxation technique to construct a relaxed SAV (RSAV) approach to improve the accuracy and consistency noticeably.

In this paper, we first propose a novel technique to construct a new exponential scalar auxiliary variable approach. By introducing a new exponential SAV, we use $k$-step backward differentiation formula (BDF$k$) to construct high-order energy stable schemes. Furthermore, the new proposed E-SAV method only needs to solve one linear equation with constant coefficients at each time step. Meanwhile, based on the exponential function, the new E-SAV approach can remove the assumptions of bounded-from-below for the free energy potential or nonlinear term. Besides, to improve its accuracy and consistency noticeably, we apply the relaxation technique which was considered in \cite{jiang2022improving} to propose a relaxed E-SAV (RE-SAV) method for gradient flows. The relaxation technique can improve the accuracy and eliminate the potential failure caused by the exponential growth.

The paper is organized as follows. In Sect.2, we provide a brief review of the SAV-type approaches. In Sect.2, we consider a new procedure to obtain a new energy stable exponential SAV (E-SAV) approach and construct first-order scheme in time. In Sect.4, some high-order unconditionally energy stable schemes with $k$-step backward differentiation formula are constructed. In Sect.5, we consider a relaxation technique to construct a relaxed E-SAV (RE-SAV) method to improve accuracy and consistency noticeably. Finally, in Sect.6, various 2D numerical simulations are demonstrated to verify the accuracy and efficiency of our proposed schemes.

\section{The review of the SAV-type approaches}
In this section, in order to show and give a comparative study for our new E-SAV approach, we provide below a brief review of the new SAV approach in \cite{huang2020highly} and the traditional E-SAV method in \cite{liu2020exponential} to construct energy stable schemes for gradient flows.
\subsection{The new SAV approach}
Assume that the energy $E(\phi)$ is bounded from below which means that there is a constant $C_0>0$ to satisfy $E(\phi)+C_0>0$. we then introduce the following scalar auxiliary variable to obtain an equivalent gradient flow model
\begin{equation}\label{intro-e3}
\aligned
R(t)=E(\phi)+C=\frac12(\phi,\mathcal{L}\phi)+\int_\Omega F(\phi)d\textbf{x}+C,
\endaligned
\end{equation}
where $C\geq C_0$ is a chosen scalar such that $R(t)>0$. Obviously, $R(t)$ can be seen as a shifted total energy. Performing integration by parts, one can find that the gradient flow system will satisfy the following energy dissipative law:
\begin{equation}\label{intro-e4}
\frac{dR(t)}{dt}=\frac{dE}{dt}=-(\mathcal{G}\mu,\mu)\leq0.
\end{equation}

Considering the definition of $R(t)$, we define a new function $\xi(t)=\frac{R(t)}{E(\phi)+C}$. It is obviously that $\xi(t)\equiv1$ at the continuous level. Other than that, to obtain the high-order numerical schemes, we introduce a new function $\theta(t)$ which can be an arbitrary function at the continuous level. Obviously, $\theta+(1-\theta)\xi\equiv1$ at the continuous level. Then, the gradient flow model \eqref{intro-e2} can be transformed into the following equivalent formulation:
\begin{equation}\label{intro-equation}
   \begin{array}{rll}
\displaystyle\frac{\partial \phi}{\partial t}&=&-\mathcal{G}\mu,\\
\mu&=&\mathcal{L}\phi+[\theta+(1-\theta)\xi]F'(\phi),\\
\xi(t)&=&\displaystyle\frac{R(t)}{E(\phi)+C},\\
\displaystyle\frac{d R}{d t}&=&\displaystyle\xi(\mu,\Delta\mu).
   \end{array}
  \end{equation}
We discretize the nonlinear term $F'(\phi)$ and $\theta$ explicitly and discretize $\phi$, $\mu$, $R$ and $\xi$ implicitly, then couple with $k$-step backward differentiation formula (BDF$k$),  the high-order unconditionally energy stable schemes can be constructed as follows:
\begin{equation}\label{intro-e5}
   \begin{array}{l}
\displaystyle\frac{\alpha\phi^{n+1}-\left[\theta^n+(1-\theta^n)\xi^{n+1}\right]\widehat{\phi}^{n}}{\Delta t}=-\mathcal{G}\mu^{n+1},\\
\mu^{n+1}=\displaystyle\mathcal{L}\phi^{n+1}+\left[\theta^n+(1-\theta^n)\xi^{n+1}\right]F'(\phi^{\ast,n+1}),\\
\xi^{n+1}=\displaystyle\frac{R^{n+1}}{E(\overline{\phi}^{n+1})+C},\\
\displaystyle\frac{R^{n+1}-R^n}{\Delta t}=-\xi^{n+1}(\mathcal{G}\overline{\mu}^{n+1},\overline{\mu}^{n+1}),
   \end{array}
\end{equation}
where $\theta^n=1+O(\Delta t^k)$ can be obtained from \cite{huang2020highly}. Here, $\alpha$, $\widehat{\phi}^{n}$ and $\phi^{\ast,n+1}$ in equation \eqref{intro-e5} are defined as follows:\\
BDF2:
\begin{equation}\label{intro-e6}
\alpha=\frac{3}{2},\quad\widehat{\phi}^{n}=2\phi^n-\frac{1}{2}\phi^{n-1},\quad\phi^{\ast,n+1}=2\phi^n-\phi^{n-1}.
\end{equation}
BDF3:
\begin{equation}\label{intro-e7}
\alpha=\frac{11}{6},\quad\widehat{\phi}^{n}=3\phi^n-\frac{3}{2}\phi^{n-1}+\frac13\phi^{n-2},\quad\phi^{\ast,n+1}=3\phi^n-3\phi^{n-1}+\phi^{n-2}.
\end{equation}
Some other high-order BDF$k$ ($k\geq4$) can be obtained from \cite{huang2020highly}. It's not difficult to prove that the BDF$k$ scheme \eqref{intro-e5} has $k$-th accuracy for $\phi$. By giving an arbitrary function $\theta(t)$ to satisfy $\theta=1+O(\Delta t^k)$, we can direct to observe that
\begin{equation}\label{intro-e8}
\displaystyle\frac{\alpha\phi^{n+1}-\left[\theta^n+(1-\theta^n)\xi^{n+1}\right]\widehat{\phi}^{n}}{\Delta t}=
\displaystyle\frac{\alpha\phi^{n+1}-\widehat{\phi}^{n}}{\Delta t}+\frac{(1-\theta^n)(1-\xi^{n+1})}{\Delta t}=\left.\frac{\partial \phi}{\partial t}\right|^{n+1}+O(\Delta t^k).
\end{equation}
Combining the first two equations in \eqref{intro-e6}, we can obtain the following linear matrix equation
\begin{equation*}
\aligned
(\alpha I+\Delta t\mathcal{G}\mathcal{L})\phi^{n+1}=\left[\theta^n+(1-\theta^n)\xi^{n+1}\right](\widehat{\phi}^n-\Delta t\mathcal{G}F'(\phi^{\ast,n+1})),
\endaligned
\end{equation*}
which means that we only require solving one linear equation with constant coefficients (see more details in \cite{huang2020highly}).

By introducing a new SAV $\theta$, the new SAV approach enjoys the following remarkable properties: (1) it only requires solving one linear system with constant coefficients at each time step; (2) it only requires the energy functional $E(\phi)$ be bounded from below; (3) it is extendable to higher-order BDF type energy stable schemes.
\subsection{The traditional E-SAV approach}
Introduce an exponential scalar auxiliary variable (E-SAV) as follows:
\begin{equation}\label{esav-e1}
\aligned
r(t)=\exp\left(E_1(\phi)\right)=\exp\left(\int_\Omega F(\phi)d\textbf{x}\right)>0.
\endaligned
\end{equation}
It is obviously $r(t)>0$ for any $t$. Then, the nonlinear functional $F'(\phi)$ in \eqref{intro-e2} can be transformed as the following equivalent formulation:
\begin{equation*}
F'(\phi)=\frac{r}{r}F'(\phi)=\frac{r}{\exp\left(E_1(\phi)\right)}F'(\phi).
\end{equation*}
Then, we have
\begin{equation*}
\aligned
&r_t=\displaystyle r\int_{\Omega}{F'}(\phi)\phi_td\textbf{x}.
\endaligned
\end{equation*}
Noting that $r(t)>0$ for any $t$, Then we have
\begin{equation}\label{esav-e2}
\aligned
&\frac{d\ln r}{dt}=\int_{\Omega}{F'}(\phi)\phi_td\textbf{x}=\frac{r}{\exp\left(E_1(\phi)\right)}\int_{\Omega}{F'(\phi)}\phi_td\textbf{x}.
\endaligned
\end{equation}
Then, the gradient flow system \eqref{intro-e2} can be transformed as follows:
\begin{equation}\label{esav-e3}
  \left\{
   \begin{array}{rll}
\displaystyle\frac{\partial \phi}{\partial t}&=&-\mathcal{G}\mu,\\
\mu&=&\displaystyle\mathcal{L}\phi+\frac{r}{\exp\left(E_1(\phi)\right)}F'(\phi),\\
\displaystyle\frac{d\ln r}{dt}&=&\displaystyle\frac{r}{\exp\left(E_1(\phi)\right)}\int_{\Omega}{F'(\phi)}\phi_td\textbf{x}.
   \end{array}
   \right.
\end{equation}
The above equivalent system satisfies the following energy dissipation law:
\begin{equation*}
\frac{d}{dt}\left[\frac12(\mathcal{L}\phi,\phi)+\ln(r)\right]=-(\mathcal{G}\mu,\mu)\leq0.
\end{equation*}
For the sake of simplicity, we only give the following first-order semi-implicit scheme:
\begin{equation*}
  \left\{
   \begin{array}{rll}
\displaystyle\frac{\phi^{n+1}-\phi^{n}}{\Delta t}&=&-\mathcal{G}\mu^{n+1},\\
\mu^{n+1}&=&\displaystyle\mathcal{L}\phi^{n+1}+\frac{r^{n+1}}{\exp\left(E_1(\phi^n)\right)}F'(\phi^n),\\
\displaystyle\frac{\ln(r^{n+1})-\ln(r^n)}{\Delta t}&=&\displaystyle\frac{r^{n+1}}{\exp\left(E_1(\phi^n)\right)}\left(F'(\phi^n),\frac{\phi^{n+1}-\phi^{n}}{\Delta t}\right).
   \end{array}
   \right.
\end{equation*}
It is not difficult to obtain the following discrete energy law:
\begin{equation*}
\aligned
\frac{1}{\Delta t}\left[E_{1st}^{n+1}-E^{n}_{1st}\right]\leq-(\mathcal{G}\mu^{n+1},\mu^{n+1})-\frac{1}{2\Delta t}(\phi^{n+1}-\phi^n,\mathcal{L}(\phi^{n+1}-\phi^n))\leq0,
\endaligned
\end{equation*}
where $E_{1st}^{n}=\frac12(\phi^n,\mathcal{L}\phi^{n})+\ln(r^n).$
\section{A new exponential SAV approach}
The new SAV approach in \cite{huang2020highly} has to introduce an extra function $\theta(t)$ to construct high-order energy stable schemes. The explicit treating of $\theta$ is not a good enough choice for the discretization of the nonlinear term $[\theta+(1-\theta)\xi]F'(\phi)$. Besides, we need to assume that the energy $E(\phi)$ is bounded from below which means that there is a constant $C_0>0$ to satisfy $E(\phi)+C_0>0$. In this section, we will consider a new technique to modify this method.

Before giving a detailed introduction, we let $N>0$ be a positive integer and set
\begin{equation*}
\Delta t=T/N,\quad t^n=n\Delta t,\quad \text{for}\quad n\leq N.
\end{equation*}

As we all know, exponential function is a special function that keeps the positive property. Thus, we can introduce the following new exponential scalar auxiliary variable:
\begin{equation}\label{exesav-e1}
\aligned
R(t)=\exp\left(E(\phi)\right)=\exp\left((\phi,\mathcal{L}\phi)+\int_\Omega F(\phi)d\textbf{x}\right).
\endaligned
\end{equation}
It is obviously $R(t)>0$ for any $t$. Now it's easy to obtain the following modified energy dissipation law:
\begin{equation*}
\frac{dR}{dt}=R\frac{dE}{dt}=-R(\mathcal{G}\mu,\mu)=-\exp(E(\phi))(\mathcal{G}\mu,\mu)\leq0.
\end{equation*}
Noticing that $\ln(R)=\ln(\exp(E(\phi)))=E(\phi)$, we can obtain the original energy dissipation law:
\begin{equation*}
\frac{dE}{dt}=\frac{d\ln(R)}{dt}=\frac{1}{R}\frac{dR}{dt}=-(\mathcal{G}\mu,\mu)\leq0.
\end{equation*}

Define $\xi=\frac{R}{\exp(E(\phi))}$. It can be easily obtained that $\xi\equiv1$ at the continuous level. Next we will introduce a new functional $U(\xi)$ to obtain the high-order approximation of $1$. $U(\xi)$ can be chosen as many formulas such as $U_2(\xi)=\xi(2-\xi)$ or $U_3(\xi)=(2-\xi)(\xi^2-\xi+1)$. Based on the exponential SAV $R(t)$ and the introduced function $U(\xi)$, the gradient flow \eqref{intro-equation} can be rewritten as the following equivalent system:
\begin{equation}\label{exesav-equation}
   \begin{array}{rll}
\displaystyle\frac{\partial \phi}{\partial t}&=&-\mathcal{G}\mu,\\
\mu&=&\mathcal{L}\phi+U(\xi)F'(\phi),\\
\xi(t)&=&\displaystyle\frac{R(t)}{\exp(E(\phi))},\\
\displaystyle\frac{d R}{d t}&=&\displaystyle-R(\mu,\mathcal{G}\mu).
   \end{array}
  \end{equation}

A first-order scheme for solving above system \eqref{exesav-equation} can be readily derived by the first-order backward Euler method as follows:
\begin{equation}\label{exesav-first-e1}
   \begin{array}{l}
\displaystyle\frac{\phi^{n+1}-U(\xi^{n+1})\phi^n}{\Delta t}=-\mathcal{G}\mu^{n+1},\\
\mu^{n+1}=\mathcal{L}\phi^{n+1}+U(\xi^{n+1})F'(\phi^n),\\
\xi^{n+1}=\displaystyle\frac{R^{n+1}}{\exp(E(\overline{\phi}^{n+1}))},\\
\displaystyle\frac{R^{n+1}-R^n}{\Delta t}=-R^{n+1}(\mathcal{G}\overline{\mu}^{n+1},\overline{\mu}^{n+1}),\\
U(\xi^{n+1})=\xi^{n+1}(2-\xi^{n+1}),
   \end{array}
\end{equation}
with the initial conditions
\begin{equation*}
\phi^0=\phi_0(x,t),\quad R^0=\exp(E(\phi^0)).
\end{equation*}

In \eqref{exesav-first-e1}, we give a first-order scheme for the variable $R$ which means $R^{n+1}=R(t^{n+1})+O(\Delta t)=1+O(\Delta t)$, we then obtain
\begin{equation*}
\xi^{n+1}=\xi(t^{n+1})+C_1\Delta t=1+C_1\Delta t.
\end{equation*}
Then, we can obtain the following equation
\begin{equation*}
U(\xi^{n+1})=\xi^{n+1}(2-\xi^{n+1})=(1+C_1\Delta t)(1-C_1\Delta t)=1-C_1^2\Delta t^2.
\end{equation*}
which means $1-U(\xi^{n+1})=O(\Delta t^2)$.

In this way, it is direct to observe that
\begin{equation}\label{exesav-first-e2}
\displaystyle\frac{\phi^{n+1}-U(\xi^{n+1})\phi^n}{\Delta t}=\frac{\phi^{n+1}-\phi^n}{\Delta t}+\frac{[1-U(\xi^{n+1})]\phi^n}{\Delta t}=\left.\frac{\partial\phi}{\partial t}\right|^{n+1}+O(\Delta t).
\end{equation}

Combining the first two equations in \eqref{exesav-first-e1} leads to the following linear equation
\begin{equation}\label{esav-first-e3}
(I+\Delta t\mathcal{G}\mathcal{L})\phi^{n+1}=U(\xi^{n+1})\left[\phi^n-\Delta t\mathcal{G}F'(\phi^n)\right].
\end{equation}
If we set
\begin{equation*}
\phi^{n+1}=U(\xi^{n+1})\overline{\phi}^{n+1}.
\end{equation*}
Then, we can compute $\overline{\phi}^{n+1}$ directly by using $\phi^n$ only:
\begin{equation}\label{exesav-first-e4}
(I+\Delta t\mathcal{G}\mathcal{L})\overline{\phi}^{n+1}=\left[\phi^n-\Delta t\mathcal{G}F'(\phi^n)\right].
\end{equation}

Next, we can compute $\xi^{n+1}$ and $R^{n+1}$ from the third and fourth equations in \eqref{exesav-first-e1} by giving the following definition:
\begin{equation*}
\overline{\mu}^{n+1}=\mathcal{L}\overline{\phi}^{n+1}+F'(\overline{\phi}^{n+1}).
\end{equation*}

Naturally, $R^{n+1}$ and $\xi^{n+1}$ can be solved out step-by-step by the following equations:
\begin{equation}\label{exesav-first-e5}
\displaystyle R^{n+1}=\frac{R^n}{1+\Delta t(\mathcal{G}\overline{\mu}^{n+1},\overline{\mu}^{n+1})},\quad\xi^{n+1}=\frac{R^{n+1}}{\exp(E(\overline{\phi}^{n+1}))}.
\end{equation}
Then, $\phi^{n+1}=U_k(\xi^{n+1})\overline{\phi}^{n+1}$ can be solved out immediately.

Let's see how the new method differs from the proposed method in \cite{huang2020highly}. In this new E-SAV scheme, We introduce a new functional $U(\xi)$ to replace $\theta+(1-\theta)\xi$. A benefit of this change is that $U(\xi)$ can be treated totally implicit which avoids the unknown problems of explicit discretization of $\theta$. Meanwhile, the two schemes have the same computational costs. To summarize, the first-order scheme \eqref{exesav-first-e1} can be implemented as follows:
\begin{enumerate}
  \item[] $\bullet$ compute $\overline{\phi}^{n+1}$ from the linear equation \eqref{exesav-first-e4};
  \item[] $\bullet$ set $\overline{\mu}^{n+1}=\mathcal{L}\overline{\phi}^{n+1}+F'(\overline{\phi}^{n+1})$ and compute $R^{n+1}$ and $\xi^{n+1}$ from \eqref{exesav-first-e5};
  \item[] $\bullet$ update $\phi^{n+1}=\xi^{n+1}(2-\xi^{n+1})\overline{\phi}^{n+1}$ and go to the next time step.
\end{enumerate}
The first-order new E-SAV scheme \eqref{exesav-first-e1} can save half CPU times compared with the traditional SAV scheme. We observe that $\phi^{n+1}$ and $\xi^{n+1}$ can be solved step by step which means that the above procedure only requires solving one linear equation with constant coefficients as in the standard semi-implicit scheme. As for the energy stability, we have the following theorem.
\begin{theorem}\label{exesav-th1}
Given $R^n>0$, we then obtain $R^{n+1}>0$. The scheme \eqref{exesav-first-e1} for the equivalent phase field system \eqref{exesav-equation} is unconditionally energy stable in the sense that
\begin{equation*}
\aligned
\displaystyle R^{n+1}=\frac{R^n}{1+\Delta t(\mathcal{G}\overline{\mu}^{n+1},\overline{\mu}^{n+1})}\leq R^n.
\endaligned
\end{equation*}
and more importantly we have
\begin{equation*}
\aligned
\ln R^{n+1}-\ln R^n\leq0.
\endaligned
\end{equation*}
\end{theorem}
\begin{proof}
Using the definition of $R(t)$, we can obtain $R^0=\exp(\phi^0)>0$. Assuming that $R^n>0$ for all $n=1,2,\ldots,K$, then we will prove $R^{K+1}>0$. From \eqref{exesav-first-e5}, we can obtain
\begin{equation*}
\aligned
\displaystyle R^{K+1}=\frac{R^K}{1+\Delta t(\mathcal{G}\overline{\mu}^{K+1},\overline{\mu}^{K+1})}
\endaligned
\end{equation*}
Noticing that $(\mathcal{G}\overline{\mu}^{K+1},\overline{\mu}^{K+1})\geq0$ for any $\overline{\mu}^{K+1}$ and $R^K>0$, then we obtain $R^{K+1}>0$. By mathematical induction, we get $R^{n+1}>0$ for any $n>0$.

Next we will give a proof of energy stability. Using the inequality $1+\Delta t(\mathcal{G}\overline{\mu}^{n+1},\overline{\mu}^{n+1})$ for any $n>0$, we immediately obtain
\begin{equation*}
\aligned
\displaystyle R^{n+1}=\frac{R^n}{1+\Delta t(\mathcal{G}\overline{\mu}^{n+1},\overline{\mu}^{n+1})}\leq R^n.
\endaligned
\end{equation*}
We observe that $E(\phi)=\ln(\exp(E(\phi)))=\ln(R)$ which means $\ln(R^n)$ will be the modified energy. Noting that the logarithm function $y=\ln(x)$ is a strictly monotone increasing function and $R^n>0$, we can also obtain the following energy stability:
\begin{equation*}
\aligned
\ln R^{n+1}-\ln R^n\leq0.
\endaligned
\end{equation*}
\end{proof}
\section{The high-order E-SAV BDF$k$ scheme}
In the first-order energy stable scheme \eqref{exesav-first-e1}, we define $U(\xi)=\xi(2-\xi)$. By using the first-order approximation for $\xi$, we obtain $U(\xi^{n+1})=1-O(\Delta t^2)$ which will not influence the first-order accuracy for $\phi$. We observe that if we combine a proper functional $U(\xi)$ with the first-order $\xi^{n+1}=1+O(\Delta t)$, we can obtain the discrete formulation $U(\xi^{n+1})=1-O(\Delta t^{k+1})$ for any $k\geq1$. Then, we can achieve overall $k$th-order accuracy coupled with $k$-step  backward differentiation formula for $\phi$ by using just a first-order approximation for $R$ and $\xi$.

We set $U(\xi)=U_k(\xi)$ and discretize the nonlinear term $F'(\phi)$ explicitly and discretize $\phi$, $\mu$, $R$ and $U$ implicitly, then couple with $k$-step backward differentiation formula (BDF$k$), the high-order unconditionally energy stable schemes can be constructed as follows:
\begin{equation}\label{exhesav-e1}
   \begin{array}{l}
\displaystyle\frac{\alpha\phi^{n+1}-U_k(\xi^{n+1})\widehat{\phi}^{n}}{\Delta t}=-\mathcal{G}\mu^{n+1},\\
\mu^{n+1}=\displaystyle\mathcal{L}\phi^{n+1}+U_k(\xi^{n+1})F'(\phi^{\ast,n+1}),\\
\xi^{n+1}=\displaystyle\frac{R^{n+1}}{\exp(E(\overline{\phi}^{n+1}))},\\
\displaystyle\frac{R^{n+1}-R^n}{\Delta t}=-R^{n+1}(\mathcal{G}\overline{\mu}^{n+1},\overline{\mu}^{n+1}),
   \end{array}
\end{equation}
Here, $\alpha$, $\widehat{\phi}^{n}$, $\phi^{\ast,n+1}$ and $U_k(\xi)$ in equation \eqref{exhesav-e1} are defined as follows:\\
BDF1:
\begin{equation}\label{exhesav-e2}
\alpha=1,\quad\widehat{\phi}^{n}=\phi^n,\quad\phi^{\ast,n+1}=\phi^n,\quad U_1(\xi)=\xi(2-\xi).
\end{equation}
BDF2:
\begin{equation}\label{exhesav-e3}
\alpha=\frac{3}{2},\quad\widehat{\phi}^{n}=2\phi^n-\frac{1}{2}\phi^{n-1},\quad\phi^{\ast,n+1}=2\phi^n-\phi^{n-1},\quad U_2(\xi)=(2-\xi)(\xi^2-\xi+1).
\end{equation}
BDF3:
\begin{equation}\label{exhesav-e4}
\alpha=\frac{11}{6},\quad\widehat{\phi}^{n}=3\phi^n-\frac{3}{2}\phi^{n-1}+\frac13\phi^{n-2},\quad\phi^{\ast,n+1}=3\phi^n-3\phi^{n-1}+\phi^{n-2},\quad U_3(\xi)=\xi(2-\xi)(\xi^2-2\xi+2).
\end{equation}
BDF4:
\begin{equation}\label{exhesav-e5}
\aligned
&\alpha=\frac{25}{12},\quad\widehat{\phi}^{n}=4\phi^n-3\phi^{n-1}+\frac43\phi^{n-2}-\frac14\phi^{n-3},\quad\phi^{\ast,n+1}=4\phi^n-6\phi^{n-1}+4\phi^{n-2}-\phi^{n-3},\\ &U_4(\xi)=(2-\xi)(\xi^4-3\xi^3+4\xi^2-2\xi+1).
\endaligned
\end{equation}
\begin{lemma}
$U_k(\xi^{n+1})$ is a $(k+1)$-th approximation to $1$ for $k=1,2,3,4$.
\end{lemma}
\begin{proof}
For $k=1$, $U_1(\xi)=1+O(\Delta t^2)$ has been proved in Section 2. Now we only shall the detailed proof for $k=2,3,4$. Noting that $\xi^{n+1}=1+C_1\Delta t$, then we have
\begin{equation*}
\aligned
U_2(\xi^{n+1})
&=(2-\xi^{n+1})[(\xi^{n+1})^2-\xi^{n+1}+1]\\
&=[1-(\xi^{n+1}-1)][(\xi^{n+1}-1)^2+(\xi^{n+1}-1)+1]\\
&=(1-C_1\Delta t)[(C_1\Delta t)^2+C_1\Delta t+1]\\
&=1-C_1^3(\Delta t)^3,
\endaligned
\end{equation*}
and for $k=3$, we have
\begin{equation*}
\aligned
U_3(\xi^{n+1})
&=\xi^{n+1}(2-\xi^{n+1})[(\xi^{n+1})^2-2\xi^{n+1}+2]\\
&=[1+(\xi^{n+1}-1)][1-(\xi^{n+1}-1)][(\xi^{n+1}-1)^2+1]\\
&=(1+C_1\Delta t)(1-C_1\Delta t)[1+(C_1\Delta t)^2]\\
&=[1-C_1^2(\Delta t)^2][1+C_1^2(\Delta t)^2]\\
&=1-C_1^4(\Delta t)^4,
\endaligned
\end{equation*}
and for $k=4$, we have
\begin{equation*}
\aligned
U_4(\xi^{n+1})
&=(2-\xi^{n+1})((\xi^{n+1})^4-3(\xi^{n+1})^3+4(\xi^{n+1})^2-2\xi^{n+1}+1)\\
&=[1-(\xi^{n+1}-1)][(\xi^{n+1}-1)^4+(\xi^{n+1}-1)^3+(\xi^{n+1}-1)^2+(\xi^{n+1}-1)+1]\\
&=(1-C_1\Delta t)[1+C_1\Delta t+(C_1\Delta t)^2+(C_1\Delta t)^3+(C_1\Delta t)^4]\\
&=1-C_1^5(\Delta t)^5,
\endaligned
\end{equation*}
which completes the proof.
\end{proof}

Using above results, we can directly observe that
\begin{equation}\label{exhesav-e6}
\displaystyle\frac{\alpha\phi^{n+1}-U_k(\xi^{n+1})\widehat{\phi}^n}{\Delta t}=\frac{\alpha\phi^{n+1}-\widehat{\phi}^n}{\Delta t}+\frac{[1-U_k(\xi^{n+1})]\widehat{\phi}^n}{\Delta t}=\left.\frac{\partial\phi}{\partial t}\right|^{n+1}+O(\Delta t^k).
\end{equation}
The E-SAV BDF$k$ schemes \eqref{exhesav-e1}-\eqref{exhesav-e4} also enjoy the same stability as the first-order scheme \eqref{exesav-first-e1}, namely, we can prove the following result using exactly the same procedure.
\begin{theorem}\label{exhesav-th3}
Given $R^n>0$, we then obtain $R^{n+1}>0$. The high-order E-SAV BDF$k$ schemes \eqref{exhesav-e1}-\eqref{exhesav-e4} are all unconditionally energy stable in the sense that
\begin{equation*}
\aligned
\displaystyle R^{n+1}=\frac{R^n}{1+\Delta t(\mathcal{G}\overline{\mu}^{n+1},\overline{\mu}^{n+1})}\leq R^n.
\endaligned
\end{equation*}
and more importantly we have
\begin{equation*}
\aligned
\ln R^{n+1}-\ln R^n\leq0.
\endaligned
\end{equation*}
\end{theorem}
\begin{remark}\label{exesav-re1}
To prevent the solution "blowing up" because of the exponential function increasing rapidly, we can add a positive constant $S$ to redefine the exponential scalar auxiliary variable:
\begin{equation*}
\aligned
R(t)=\exp\left(\frac{E(\phi)}{S}\right)=\exp\left(\frac{1}{S}(\phi,\mathcal{L}\phi)+\frac{1}{S}\int_\Omega F(\phi)d\textbf{x}\right).
\endaligned
\end{equation*}
The energy dissipation law is also keep original at the continuous level although a positive constant $S$ is added to the variable $R$:
\begin{equation*}
\frac{dE}{dt}=\frac{Sd\ln(R)}{dt}=\frac{S}{R}\frac{dR}{dt}=-(\mathcal{G}\mu,\mu)\leq0.
\end{equation*}
\end{remark}
\section{The E-SAV approach with relaxation}
In the new E-SAV scheme \eqref{exhesav-e1}, notice that $\xi(t)\equiv1$ at the continuous level because of $R(t)=\exp{E(\phi)}$. Hence, the the modified energy $\ln R(t)$ for the equivalent model \eqref{exesav-equation} and the original energy $E(\phi)$ are equal in the PDE level. However, the numerical results of $R(t)$ and $\exp{E(\phi)}$ are not equal anymore, which means the discrete energies $\ln R^{n+1}$ and $E(\phi^{n+1})$ are not equivalent anymore. Inspired by the R-SAV approach
in \cite{jiang2022improving}, we construct the following E-SAV approach with relaxation (RE-SAV), which not only inherits all the advantages of the new E-SAV approach, but can also significantly improve its accuracy.

If we combine the relaxed technique with the considered semi-implicit BDF$k$ time marching method in \eqref{exhesav-e1}, we have the following high order RE-SAV BDF$k$ scheme.

\textbf{Step I}: Compute $\phi^{n+1}$ and $\widetilde{R}^{n+1}$ by the following semi-implicit E-SAV BDF$k$ scheme:
\begin{equation}\label{resav-e1}
   \begin{array}{l}
\displaystyle\frac{\alpha\overline{\phi}^{n+1}-\widehat{\phi}^{n}}{\Delta t}=-\mathcal{G}\mathcal{L}\overline{\phi}^{n+1}-\mathcal{G}F'(\phi^{\ast,n+1}),\\
\displaystyle\frac{\widetilde{R}^{n+1}-R^n}{\Delta t}=-\widetilde{R}^{n+1}(\mathcal{G}\overline{\mu}^{n+1},\overline{\mu}^{n+1}),\\
\xi^{n+1}=\displaystyle\frac{\widetilde{R}^{n+1}}{\exp(E(\overline{\phi}^{n+1}))},\\
\phi^{n+1}=U(\xi^{n+1})\overline{\phi}^{n+1},
   \end{array}
\end{equation}
where $\overline{\mu}^{n+1}=\displaystyle\mathcal{L}\overline{\phi}^{n+1}+F'(\overline{\phi}^{n+1})$. $\alpha$, $\widehat{\phi}^{n}$, $\phi^{\ast,n+1}$ and $U_k(\xi)$ in above equations can be founded in \eqref{exhesav-e2}-\eqref{exhesav-e5}.

\textbf{Step II}: Update the scalar auxiliary variable $R^{n+1}$ via a relaxation step as
\begin{equation}\label{resav-e2}
R^{n+1}=\lambda_0\widetilde{R}^{n+1}+(1-\lambda_0)\exp\left(E(\phi^{n+1})\right),\quad \lambda_0\in\mathcal{V}.
\end{equation}
Here is $\mathcal{V}$ a set defined by $\mathcal{V}=\mathcal{V}_1\cap\mathcal{V}_2$, where
\begin{equation}\label{resav-e3}
\mathcal{V}_1=\{\lambda|\lambda\in[0,1]\},
\end{equation}
\begin{equation}\label{resav-e4}
\mathcal{V}_2=\left\{\lambda|R^{n+1}-\widetilde{R}^{n+1}\leq\frac{\Delta t\kappa(\mathcal{G}\overline{\mu}^{n+1},\overline{\mu}^{n+1})R^n}{1+\Delta t(\mathcal{G}\overline{\mu}^{n+1},\overline{\mu}^{n+1})},\quad R^{n+1}=\lambda\widetilde{R}^{n+1}+(1-\lambda)\exp\left(E(\phi^{n+1})\right)\right\}.
\end{equation}
Here, $\kappa\in[0,1]$ is an artificial parameter that can be manually assigned.

The set $\mathcal{V}_2$ in \eqref{resav-e4} can be simplified as
\begin{equation}\label{resav-e5}
\mathcal{V}_2=\left\{\lambda|\left[\widetilde{R}^{n+1}-\exp\left(E(\phi^{n+1})\right)\right]\lambda\leq\left[\widetilde{R}^{n+1}-\exp\left(E(\phi^{n+1})\right)\right]+\frac{\Delta t\kappa(\mathcal{G}\overline{\mu}^{n+1},\overline{\mu}^{n+1})R^n}{1+\Delta t(\mathcal{G}\overline{\mu}^{n+1},\overline{\mu}^{n+1})}\right\}.
\end{equation}

Firstly, notice the fact $\Delta t(\mathcal{G}\overline{\mu}^{n+1},\overline{\mu}^{n+1})\geq0$ and $R^n>0$. Then, it is obviously to see that $1\in\mathcal{V}$ which means the set $\mathcal{V}$ is non-empty. Secondly, the optimal relaxation parameter $\lambda_0$ can be chosen as follows:
the optimal relaxation parameter $\lambda_0$ can be chosen as a solution of the following optimization problem:
\begin{equation}\label{resav-e6}
\aligned
\lambda_0=\min\limits_{\lambda\in[0,1]}\lambda\quad s.t.~\left[\widetilde{R}^{n+1}-\exp\left(E(\phi^{n+1})\right)\right]\lambda\leq\left[\widetilde{R}^{n+1}-\exp\left(E(\phi^{n+1})\right)\right]+\frac{\Delta t\kappa(\mathcal{G}\overline{\mu}^{n+1},\overline{\mu}^{n+1})R^n}{1+\Delta t(\mathcal{G}\overline{\mu}^{n+1},\overline{\mu}^{n+1})}.
\endaligned
\end{equation}

(1) if $\widetilde{R}^{n+1}<\exp\left(E(\phi^{n+1})\right)$, the inequality in \eqref{resav-e6} will be simplified as
\begin{equation*}
\aligned
\lambda\geq1+\frac{\Delta t\kappa(\mathcal{G}\overline{\mu}^{n+1},\overline{\mu}^{n+1})R^n}{\left[1+\Delta t(\mathcal{G}\overline{\mu}^{n+1},\overline{\mu}^{n+1})\right]\left[\widetilde{R}^{n+1}-\exp\left(E(\phi^{n+1})\right)\right]}.
\endaligned
\end{equation*}
which means $\lambda\geq1$. Then, the optimization problem in \eqref{resav-e6} can be solved as
\begin{equation}\label{resav-e7}
\aligned
\lambda_0=\max\left\{0,1-a\right\},
\endaligned
\end{equation}
where $a=\frac{\Delta t\kappa(\mathcal{G}\overline{\mu}^{n+1},\overline{\mu}^{n+1})R^n}{\left[1+\Delta t(\mathcal{G}\overline{\mu}^{n+1},\overline{\mu}^{n+1})\right]\left|\widetilde{R}^{n+1}-\exp\left(E(\phi^{n+1})\right)\right|}.$

(2) if $\widetilde{R}^{n+1}=\exp\left(E(\phi^{n+1})\right)$, any arbitrary parameter $\lambda$ between $0$ and $1$ will satisfy the inequality in \eqref{resav-e6}. Thus, we have $\lambda_0=\min\left[0,1\right]=0$.

(3) if $\widetilde{R}^{n+1}>\exp\left(E(\phi^{n+1})\right)$, The inequality in \eqref{resav-e6} will be simplified as
\begin{equation*}
\aligned
\lambda\leq1\leq1+a.
\endaligned
\end{equation*}
which means $\lambda$ can be any arbitrary parameter between $0$ and $1$. Thus, we also have $\lambda_0=\min\left[0,1\right]=0$.

In summary, we can choose the optimal relaxation parameter $\lambda_0$ as follows:
\begin{remark}\label{resav-re1}
If $\widetilde{R}^{n+1}<\exp\left(E(\phi^{n+1})\right)$, then we define $a=\frac{\Delta t\kappa(\mathcal{G}\overline{\mu}^{n+1},\overline{\mu}^{n+1})R^n}{\left[1+\Delta t(\mathcal{G}\overline{\mu}^{n+1},\overline{\mu}^{n+1})\right]\left|\widetilde{R}^{n+1}-\exp\left(E(\phi^{n+1})\right)\right|}.$ Then the optimal relaxation parameter $\lambda_0$ can be solved as
\begin{equation}\label{resav-e8}
\aligned
\lambda_0=\left\{
   \begin{array}{rll}
\max\left\{0,1-a\right\},&&\widetilde{R}^{n+1}<\exp\left(E(\phi^{n+1})\right),\\
0,&&\widetilde{R}^{n+1}\geq\exp\left(E(\phi^{n+1})\right).
   \end{array}
   \right.
\endaligned
\end{equation}
\end{remark}
\begin{theorem}\label{resav-th1}
Given $R^n>0$, we then obtain $R^{n+1}>0$. The RE-SAV BDF$k$ scheme \eqref{resav-e1}-\eqref{resav-e2} is unconditionally energy stable in the sense that
\begin{equation*}
\aligned
\displaystyle R^{n+1}\leq\widetilde{R}^{n+1}+\frac{R^n}{1+\Delta t(\mathcal{G}\overline{\mu}^{n+1},\overline{\mu}^{n+1})}\leq R^n.
\endaligned
\end{equation*}
and more importantly we have
\begin{equation*}
\aligned
\ln R^{n+1}-\ln R^n\leq0.
\endaligned
\end{equation*}
\end{theorem}
\begin{proof}
For the first step of the RE-SAV scheme \eqref{resav-e1} and the Theorem \ref{exhesav-th3}, we could get
\begin{equation}\label{resav-e9}
\aligned
\displaystyle \widetilde{R}^{n+1}=\frac{R^n}{1+\Delta t(\mathcal{G}\overline{\mu}^{n+1},\overline{\mu}^{n+1})}.
\endaligned
\end{equation}
Then for $R^n>0$, we obtain $\widetilde{R}^{n+1}>0$. Considering that $R^{n+1}=\lambda\widetilde{R}^{n+1}+(1-\lambda)\exp\left(E(\phi^{n+1})\right)$, we immediately have $R^{n+1}>0$.

From the constraint condition in \eqref{resav-e4}, we could obtain
\begin{equation}\label{resav-e10}
R^{n+1}-\widetilde{R}^{n+1}\leq\frac{\Delta t\kappa(\mathcal{G}\overline{\mu}^{n+1},\overline{\mu}^{n+1})R^n}{1+\Delta t(\mathcal{G}\overline{\mu}^{n+1},\overline{\mu}^{n+1})}.
\end{equation}
Combining the above two inequalities \eqref{resav-e9}-\eqref{resav-e10} and noting $\kappa\in[0,1]$, we could have
\begin{equation}\label{resav-e11}
\aligned
R^{n+1}
&\leq\widetilde{R}^{n+1}+\frac{\Delta t\kappa(\mathcal{G}\overline{\mu}^{n+1},\overline{\mu}^{n+1})R^n}{1+\Delta t(\mathcal{G}\overline{\mu}^{n+1},\overline{\mu}^{n+1})}\\
&=\frac{R^n}{1+\Delta t(\mathcal{G}\overline{\mu}^{n+1},\overline{\mu}^{n+1})}+\frac{\Delta t\kappa(\mathcal{G}\overline{\mu}^{n+1},\overline{\mu}^{n+1})R^n}{1+\Delta t(\mathcal{G}\overline{\mu}^{n+1},\overline{\mu}^{n+1})}\\
&=\frac{1+\Delta t\kappa(\mathcal{G}\overline{\mu}^{n+1},\overline{\mu}^{n+1})}{1+\Delta t(\mathcal{G}\overline{\mu}^{n+1},\overline{\mu}^{n+1})}R^n\\
&\leq R^n.
\endaligned
\end{equation}
Noting that the logarithm function $y=\ln(x)$ is a strictly monotone increasing function and $R^n>0$, we can also obtain the following energy stability:
\begin{equation*}
\aligned
\ln R^{n+1}-\ln R^n\leq0.
\endaligned
\end{equation*}
\end{proof}

\begin{remark}\label{resav-re2}
In numerical simulation by using the new E-SAV approach, if we set $S=1$ and use a big time step, the error between $R^{n+1}$ and $\exp(E(\phi^{n+1}))$ may increase rapidly because of the exponential growth. The proposed RE-SAV approach is a good way to control the error between $R^{n+1}$ and $\exp(E(\phi^{n+1}))$. The relaxation technique guarantees that $R^{n+1}$ is a good approximation of $\exp(E(\phi^{n+1}))$.
\end{remark}

Actually, the high-order RE-SAV BDF$k$ scheme \eqref{resav-e1}-\eqref{resav-e2} can be divided into the following three steps:

\textbf{Step I}: Compute $\overline{\phi}^{n+1}$ by the following semi-implicit BDF$k$ scheme:
\begin{equation}\label{resav-e12}
   \begin{array}{l}
\displaystyle\frac{\alpha\overline{\phi}^{n+1}-\widehat{\phi}^{n}}{\Delta t}=-\mathcal{G}\mathcal{L}\overline{\phi}^{n+1}-\mathcal{G}F'(\phi^{\ast,n+1}).
   \end{array}
\end{equation}

\textbf{Step II}: Compute $\widetilde{R}^{n+1}$ and $\xi^{n+1}$ by the known $\overline{\phi}^{n+1}$ and the following scheme:
\begin{equation}\label{resav-e13}
\displaystyle\frac{\widetilde{R}^{n+1}-R^n}{\Delta t}=-\widetilde{R}^{n+1}(\mathcal{G}\overline{\mu}^{n+1},\overline{\mu}^{n+1}),\quad\xi^{n+1}=\displaystyle\frac{\widetilde{R}^{n+1}}{\exp(E(\overline{\phi}^{n+1}))}.
\end{equation}

\textbf{Step III}: Using $\overline{\phi}^{n+1}$ and $\xi^{n+1}$ to obtain a modified $\phi^{n+1}$:
\begin{equation}\label{resav-e14}
\phi^{n+1}=U(\xi^{n+1})\overline{\phi}^{n+1}.
\end{equation}

\textbf{Step IV}: Update $R^{n+1}$ by the known $\widetilde{R}^{n+1}$ and $\phi^{n+1}$ via a relaxation step as
\begin{equation}\label{resav-e15}
R^{n+1}=\lambda_0\widetilde{R}^{n+1}+(1-\lambda_0)\exp\left(E(\phi^{n+1})\right).
\end{equation}

\section{Examples and discussion}
In this section, several numerical examples are given to demonstrate the accuracy, energy stability and efficiency of the proposed E-SAV and RE-SAV schemes when applying to some classical gradient flow models such as the Allen-Cahn, Cahn-Hilliard and Swift-Hohenberg model. In the following examples, we consider the periodic boundary conditions and use a Fourier spectral method in space. A comparative study in accuracy of the traditional SAV scheme in \cite{shen2018scalar}, the new E-SAV scheme with no constant $S$ (E-SAV1), the new E-SAV scheme with $S=10$ (E-SAV2) and RE-SAV scheme are considered to show the accuracy and efficiency.

\subsection{Cahn-Hilliard and Allen-Cahn models}
As we all know, Cahn-Hilliard and Allen-Cahn models are very classical phase field models and have been widely used in many fields involving physics, materials science, finance and image processing \cite{chen2018accurate,chen2018power,du2018stabilized}.

Consider the following gradient flow:
\begin{equation}\label{section5_e_model}
  \left\{
   \begin{array}{rlr}
\displaystyle\frac{\partial \phi}{\partial t}&=-\mathcal{G}\mu,     &(\textbf{x},t)\in\Omega\times J,\\
                                          \mu&=-\epsilon^2\Delta \phi+F'(\phi),&(\textbf{x},t)\in\Omega\times J,
   \end{array}
   \right.
\end{equation}
where $J=(0,T]$, $\mu=\frac{\delta E}{\delta \phi}$ is the chemical potential and the energy $E$ is the following Lyapunov energy functional:
\begin{equation}\label{section5_energy1}
E(\phi)=\int_{\Omega}(\frac{\epsilon^2}{2}|\nabla \phi|^2+F(\phi))d\textbf{x},
\end{equation}
where the most commonly used form Ginzburg-Landau double-well type potential is defined as $F(\phi)=\frac{1}{4}(\phi^2-1)^2$. In general, if $\mathcal{G}=I$, the above model will be Allen-Cahn model. We set $\mathcal{G}=-\Delta$, the above model will be Cahn-Hilliard model.

\textbf{Example 1}: Consider the above gradient flow in $\Omega=[0,2\pi]^2$, and the following initial condition \cite{ShenA}:
\begin{equation*}
\aligned
\phi(x,y,0)=0.5\cos(x)\cos(y).
\endaligned
\end{equation*}

For Allen-Cahn model, we adopt $(N_x,N_y)=(256,256)$, and set model parameters $T=1$, $\epsilon=0.01$ and numerical parameters $C=1$ in SAV scheme, $S=10$ in E-SAV2 scheme, $\kappa=1$ in RE-SAV scheme. We use the first-order accurate time discrete schemes for all considered numerical approaches. The computational error and convergence rates are shown in Table \ref{tab:tab1}. We observe that all convergence rates for the listed four schemes are consistent with the theoretical results. For the E-SAV1 scheme, the exponential growth affects the precision of numerical solution. An introducing constant $S=10$ for E-SAV2 scheme is efficient to improve the accuracy. The RE-SAV approach is also a good way to improve the accuracy by controlling the error between $R^{n+1}$ and $\exp(E(\phi^{n+1}))$. The numerical results for the RE-SAV BDF$k$ $(k=1,2,3,4)$ schemes are also given in Table \ref{tab:tab2}, where we can observe the expected convergence rate of the field variable $\phi$.
\begin{table}[h!b!p!]
\small
\centering
\caption{\small The $L^\infty$ errors, convergence rates for first-order scheme in time of the SAV¡¢E-SAV1¡¢E-SAV2¡¢RE-SAV approaches for Allen-Cahn equation. }\label{tab:tab1}
\begin{tabular}{|c|c|c|c|c|c|c|c|c|}
\hline
&\multicolumn{2}{c|}{SAV}&\multicolumn{2}{c|}{E-SAV1}&\multicolumn{2}{c|}{E-SAV2}&\multicolumn{2}{c|}{RE-SAV}\\
\cline{1-9}
$\Delta t$&Error&Rate&Error&Rate&Error&Rate&Error&Rate\\
\cline{1-9}
$\frac{1}{10}$      &6.9550e-2   &---   &1.4013e-1   &---    &5.0192e-2   &---   &6.3217e-2   &---   \\
$\frac{1}{20}$      &3.5750e-2   &0.9601&6.9572e-2   &1.0102 &2.5494e-2   &0.9773&2.7020e-2   &1.2262\\
$\frac{1}{40}$      &1.8114e-2   &0.9808&3.4663e-2   &1.0051 &1.2831e-2   &0.9905&1.2934e-2   &1.0628\\
$\frac{1}{80}$      &9.0998e-3   &0.9932&1.7289e-2   &1.0035 &6.4176e-3   &0.9995&6.3816e-3   &1.0192\\
$\frac{1}{160}$     &4.5421e-3   &1.0025&8.6193e-3   &1.0042 &3.1901e-3   &1.0084&3.1607e-3   &1.0137\\
\hline
\end{tabular}
\end{table}
An energies comparison of the considered SAV, E-SAV1, E-SAV2 and RE-SAV methods in solving the Allen-Cahn equation is shown in Figure \ref{fig:fig0} and Figure \ref{fig:fig01}. From Figure \ref{fig:fig0}(a) and Figure \ref{fig:fig01}(a), we find that the RE-SAV scheme provides accurate result than the SAV scheme. From Figure \ref{fig:fig0}(b) and Figure \ref{fig:fig01}(b), we observe that the RE-SAV scheme provides less error between $R(t)$ and $\exp(E(\phi))$ than the new E-SAV scheme. It means the relaxation step increases the numerical accuracy and guarantees the numerical consistency between $R(t)$ and $\exp(E(\phi))$.
\begin{figure}[htp]
\centering
\subfigure[Log-log plot for the energy evolution]{
\includegraphics[width=8cm,height=8cm]{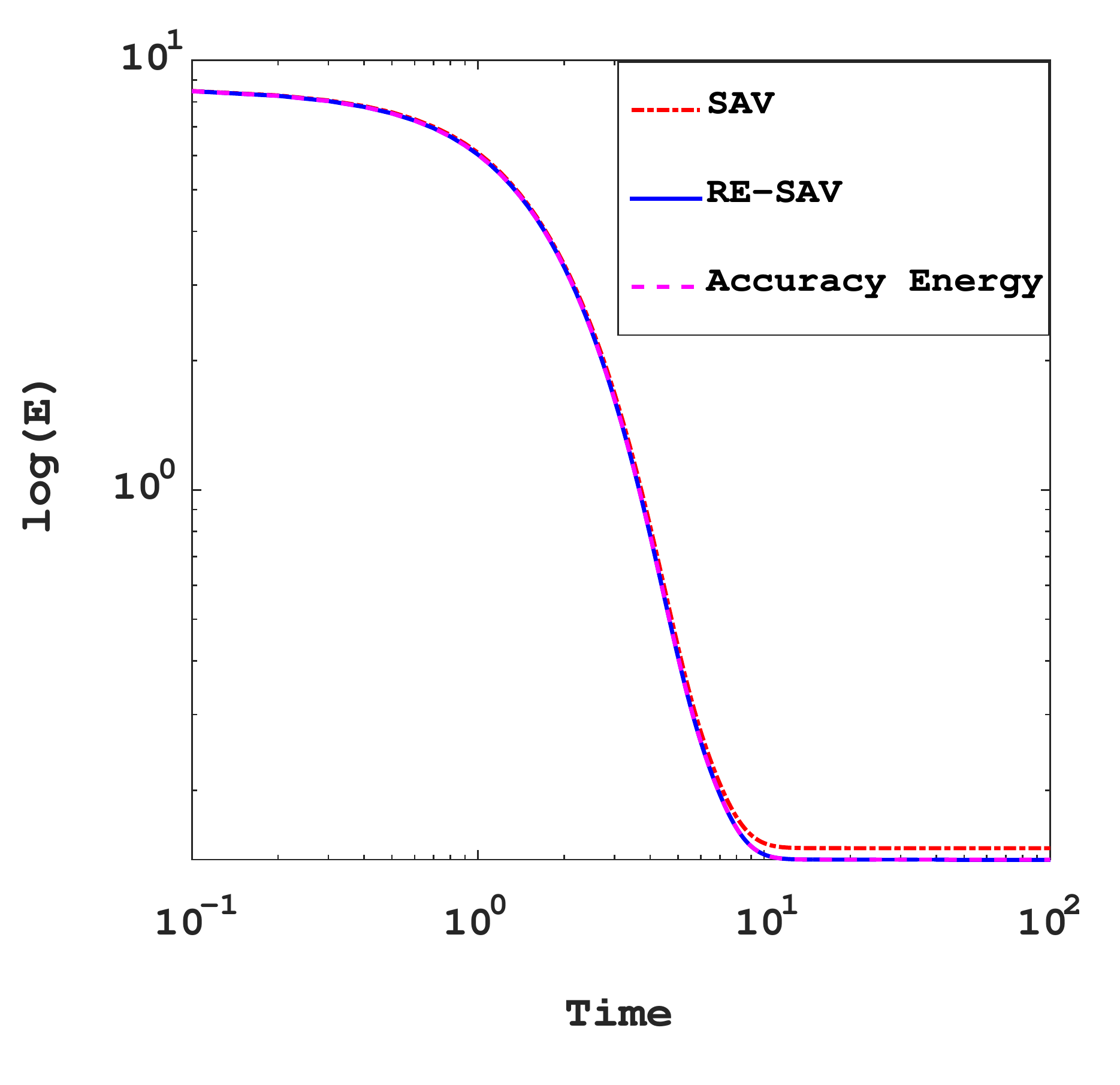}}
\subfigure[Log-log plot for the E-SAV $R(t)$ evolution]{
\includegraphics[width=8cm,height=8cm]{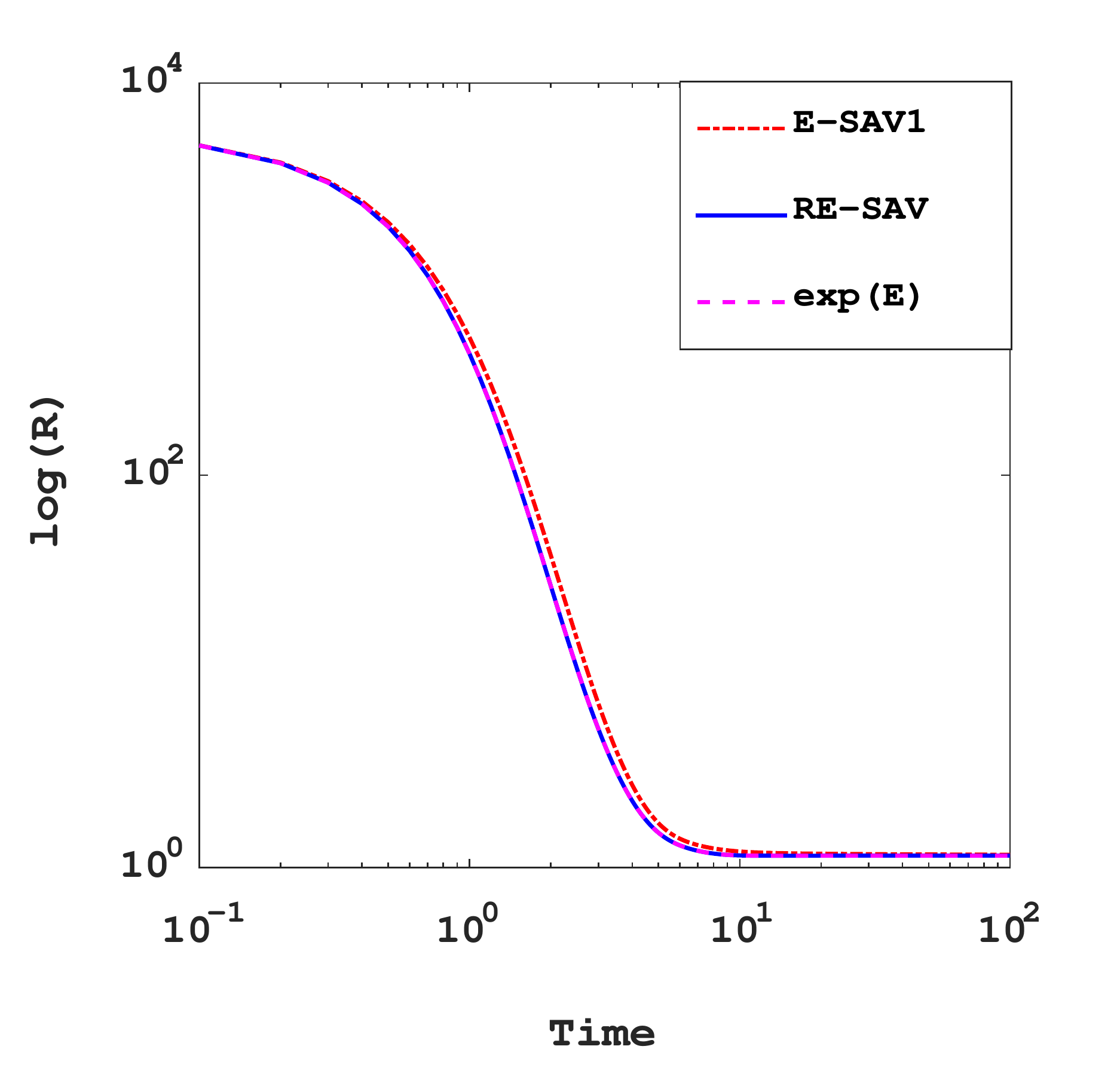}}
\caption{A comparison of the SAV, E-SAV1 and RE-SAV methods in solving the Allen-Cahn equation. (a) the numerical energies
using the first-order SAV and the RE-SAV schemes with $\Delta t=0.01$. (b) the introducing E-SAV $R(t)$ for the E-SAV1 and RE-SAV schemes with $\Delta t=0.1$.}\label{fig:fig0}
\end{figure}
\begin{figure}[htp]
\centering
\subfigure[Numerical results of $\mathcal{E}-E(\phi)$]{
\includegraphics[width=8cm,height=8cm]{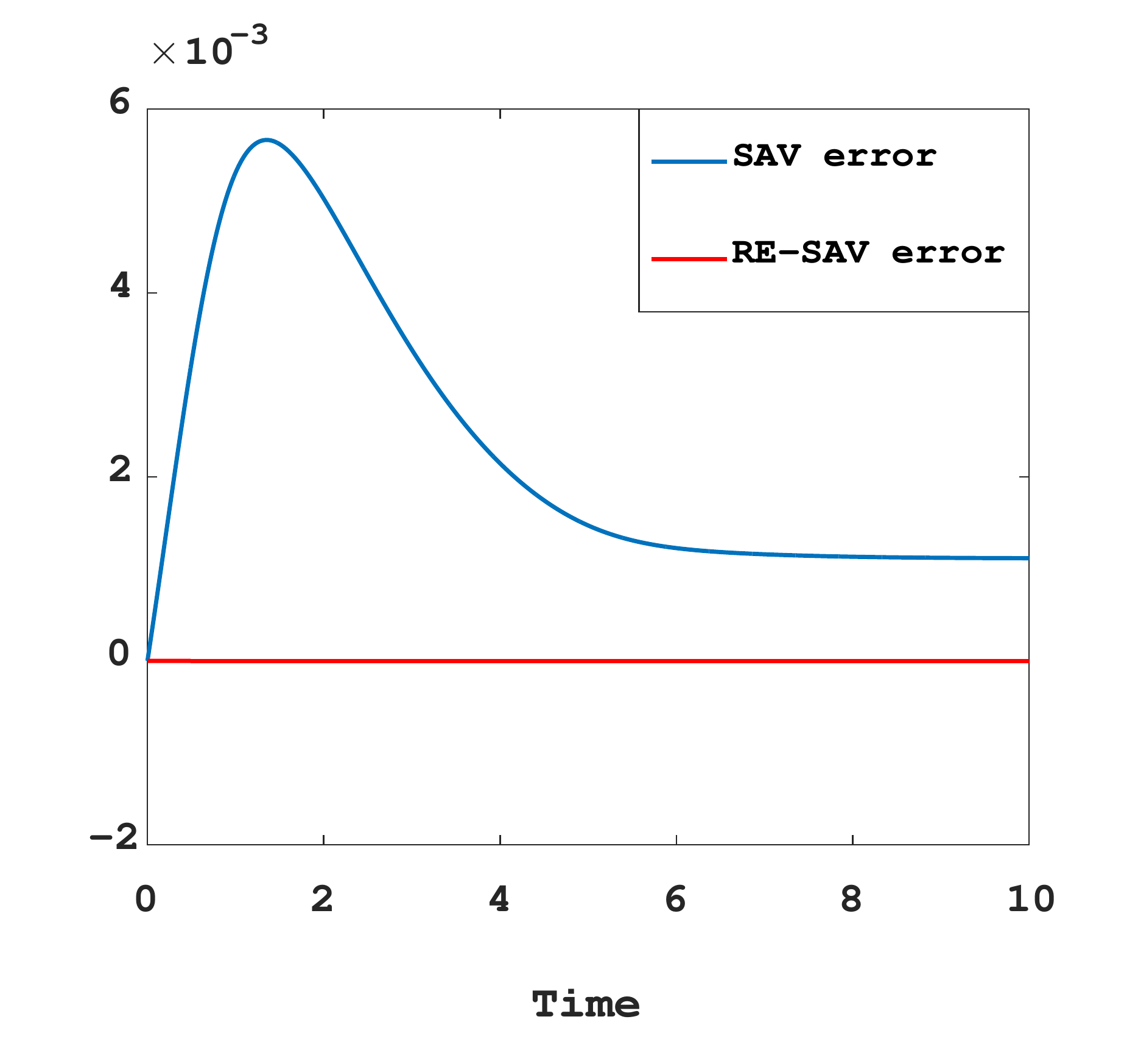}}
\subfigure[Numerical results of $R(t)-\exp(E(\phi/S))$]{
\includegraphics[width=8cm,height=8cm]{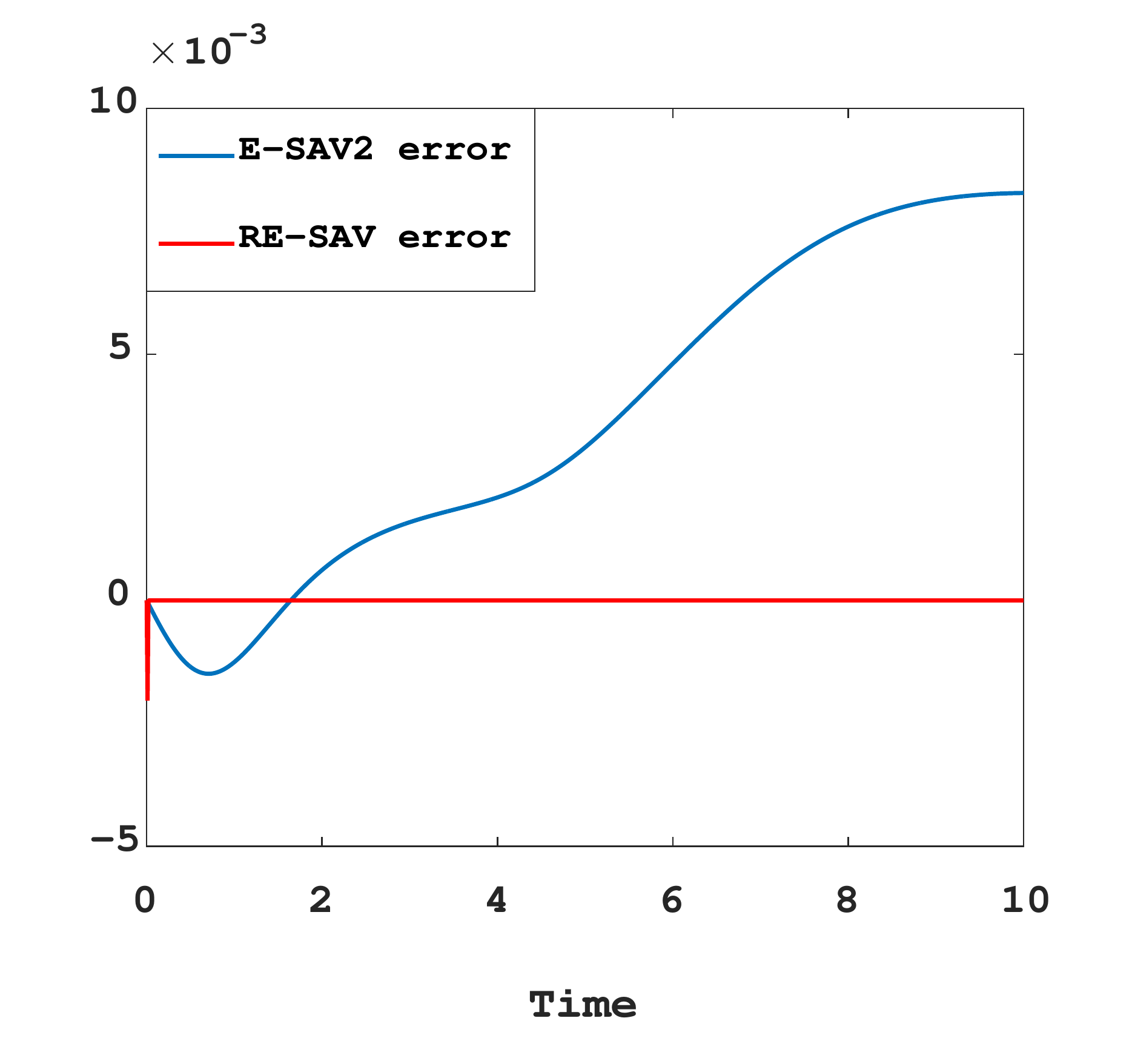}}
\caption{A comparison of the SAV, E-SAV2 and RE-SAV methods in solving the Allen-Cahn equation. (a) Numerical results of $\mathcal{E}-E(\phi)$ using the first-order SAV and the RE-SAV schemes with $\Delta t=0.01$. (b) Numerical results of $R(t)-\exp(E(\phi/S))$ for the E-SAV2 and RE-SAV schemes with $\Delta t=0.01$.}\label{fig:fig01}
\end{figure}

For Cahn-Hilliard model, we adopt uniform meshes $N_x=N_y=256$ and $T=1$, $\epsilon^2=0.16$. Given the analytical solutions are unknown, we calculate the error as the difference between the numerical solutions using the current time step and the numerical solutions using the adjacent finer time step $\Delta t_{ref}=0.0001$. Table \ref{tab:tab3} shows the results of the errors and convergence rates for the RE-SAV BDF$k$ $(k=1,2,3,4)$ scheme. Numerical results demonstrate the accuracy and efficiency of our proposed scheme.
\begin{table}[h!b!p!]
\small
\centering
\caption{\small The $L^\infty$ errors, convergence rates for the RE-SAV BDF$k$ $(k=1,2,3,4)$ schemes of Allen-Cahn equation.}\label{tab:tab2}
\begin{tabular}{|c|c|c|c|c|c|c|c|c|}
\hline
RE-SAV&\multicolumn{2}{c|}{BDF1}&\multicolumn{2}{c|}{BDF2}&\multicolumn{2}{c|}{BDF3}&\multicolumn{2}{c|}{BDF4}\\
\cline{1-9}
$\Delta t$&Error&Rate&Error&Rate&Error&Rate&Error&Rate\\
\cline{1-9}
$\frac14$      &4.6900e-2   &---   &5.0622e-2   &---    &1.0612e-2   &---   &2.6589e-3   &---   \\
$\frac18$      &1.3850e-2   &1.7597&1.5827e-2   &1.6774 &1.8173e-3   &2.5458&3.3441e-4   &2.9911\\
$\frac{1}{16}$ &5.4740e-3   &1.3392&4.3988e-3   &1.8472 &2.6317e-4   &2.7877&2.0908e-5   &3.9949\\
$\frac{1}{32}$ &2.5213e-3   &1.1184&1.1561e-3   &1.9278 &3.5325e-5   &2.8972&1.4242e-6   &3.8758\\
$\frac{1}{64}$ &1.1999e-3   &1.0713&2.9531e-4   &1.9690 &4.5726e-6   &2.9496&9.4212e-8   &3.9181\\
$\frac{1}{128}$&5.5727e-4   &1.1064&7.3797e-5   &2.0006 &5.8058e-7   &2.9745&6.0654e-9   &3.9572\\
\hline
\end{tabular}
\end{table}
\begin{table}[h!b!p!]
\small
\centering
\caption{\small The $L^\infty$ errors, convergence rates for the BDF$k$ $(k=1,2,3,4)$ modified E-SAV schemes of Cahn-Hilliard equation with $T=1$, $\epsilon^2=0.16$ and $\Delta t_{ref}=0.0001$. }\label{tab:tab3}
\begin{tabular}{|c|c|c|c|c|c|c|c|c|}
\hline
RE-SAV&\multicolumn{2}{c|}{BDF1}&\multicolumn{2}{c|}{BDF2}&\multicolumn{2}{c|}{BDF3}&\multicolumn{2}{c|}{BDF4}\\
\cline{1-9}
$\Delta t$&Error&Rate&Error&Rate&Error&Rate&Error&Rate\\
\cline{1-9}
$\frac18$      &8.2460e-1   &---   &1.9525e-1   &---    &6.7122e-2   &---   &1.5909e-2   &---   \\
$\frac{1}{16}$ &3.5383e-1   &1.2206&4.4761e-2   &2.1250 &6.8158e-3   &3.2998&8.5741e-4   &4.2137\\
$\frac{1}{32}$ &1.5361e-1   &1.2038&1.0606e-2   &2.0774 &7.5492e-4   &3.1745&4.1014e-5   &4.3858\\
$\frac{1}{64}$ &7.0724e-2   &1.1190&2.5847e-3   &2.0368 &8.8972e-5   &3.0849&2.4357e-6   &4.0737\\
$\frac{1}{128}$&3.3726e-2   &1.0683&6.3853e-4   &2.0172 &1.0825e-5   &3.0389&1.5565e-7   &3.9679\\
$\frac{1}{256}$&1.6304e-2   &1.0486&1.5868e-4   &2.0086 &1.3382e-6   &3.0160&1.0484e-8   &3.8920\\
\hline
\end{tabular}
\end{table}

\textbf{Example 2}: In the following, we solve a benchmark problem for the merging of a rectangular array of $9\times9$ circles governed by Cahn-Hilliard equation on $[0,2)^2$ which can also be seen in \cite{huang2020highly}. To give a more efficient simulation, we specify the operators $\mathcal{L}=-\epsilon^2\Delta+\beta I$ and $F(\phi)=\frac{1}{4}(\phi^2-1-\beta)^2$. We take $\epsilon=0.01$, $\beta=2$ and discretize the space by the Fourier spectral method with $256\times256$ modes. The initial condition is chosen as the following
\begin{equation*}
\aligned
\phi_0(x,y,0)=80-\sum\limits_{i=1}^9\sum\limits_{j=1}^9\displaystyle\frac{\tanh(\sqrt{(x-x_i)^2+(y-y_j)^2}-R_0)}{\sqrt{2}\epsilon},
\endaligned
\end{equation*}
where $R_0=0.085$, $x_i=0.2\times i$ and $x_j=0.2\times j$ for $i,j=1,\ldots9$.

Snapshots of the phase variable $\phi$ taken at $t=0$, $0.5$, $1$, $3$, $4.2$, $4.8$, $10$ and $100$ with $\Delta t=0.01$ are shown in Figure \ref{fig:fig2}. The phase separation and coarsening process can be observed very simply which is consistent with the results in \cite{huang2020highly}. In Figure \ref{fig:fig3}, we plot the time evolution of the energy functional with three different time step size of $\Delta t=0.01$, $0.1$ and $1$ by using the first-order RE-SAV scheme. Meanwhile, we plot the energy evolution for the traditional SAV approach and the proposed RE-SAV approach with $\Delta t=0.1$. All energy curves show the monotonic decays for all time steps which confirms that the algorithm is unconditionally energy stable. Furthermore, the RE-SAV method provides more accurate results.
\begin{figure}[htp]
\centering
\subfigure[t=0]{
\includegraphics[width=3.8cm,height=3.8cm]{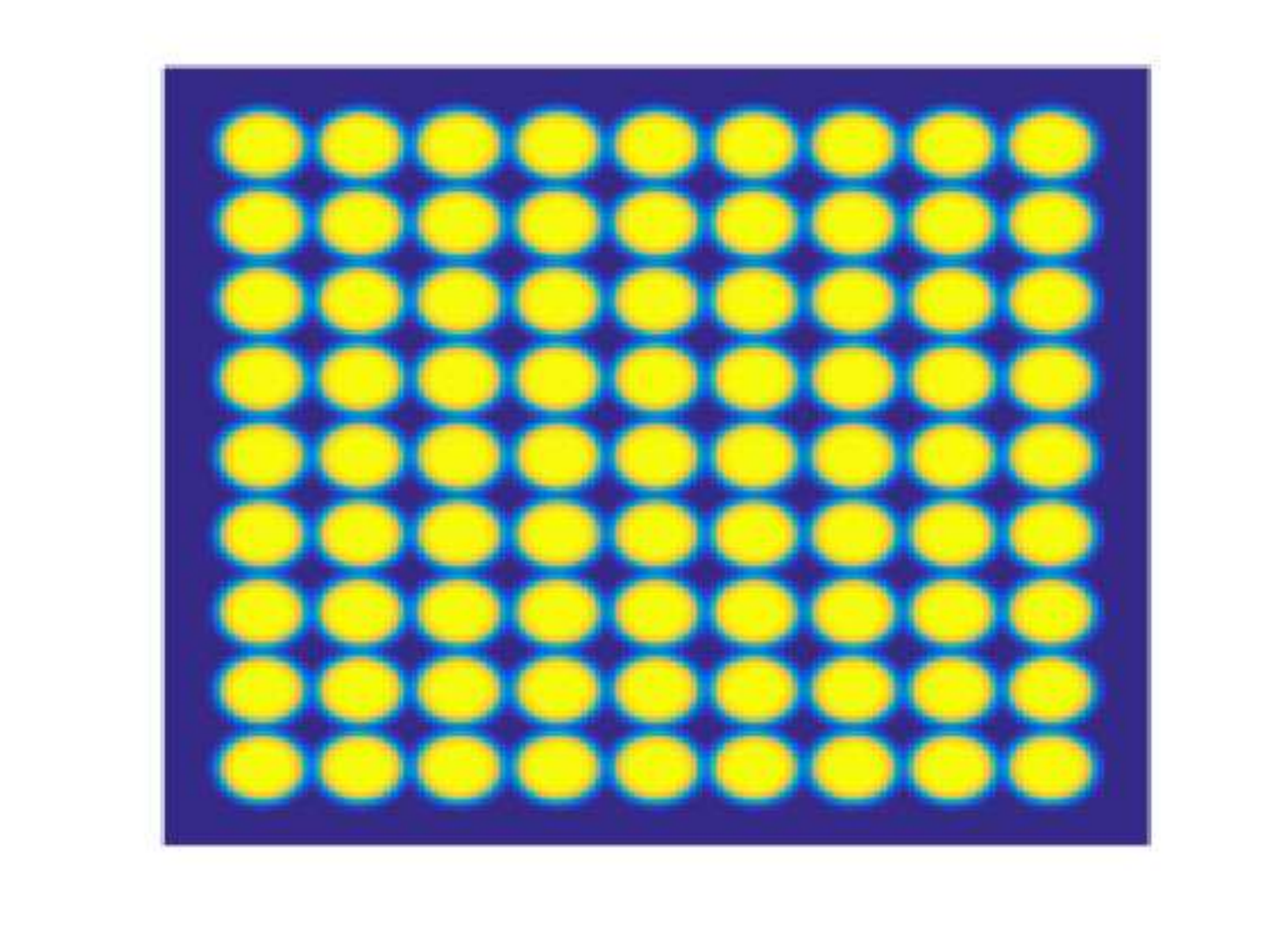}
}
\subfigure[t=0.5]
{
\includegraphics[width=3.8cm,height=3.8cm]{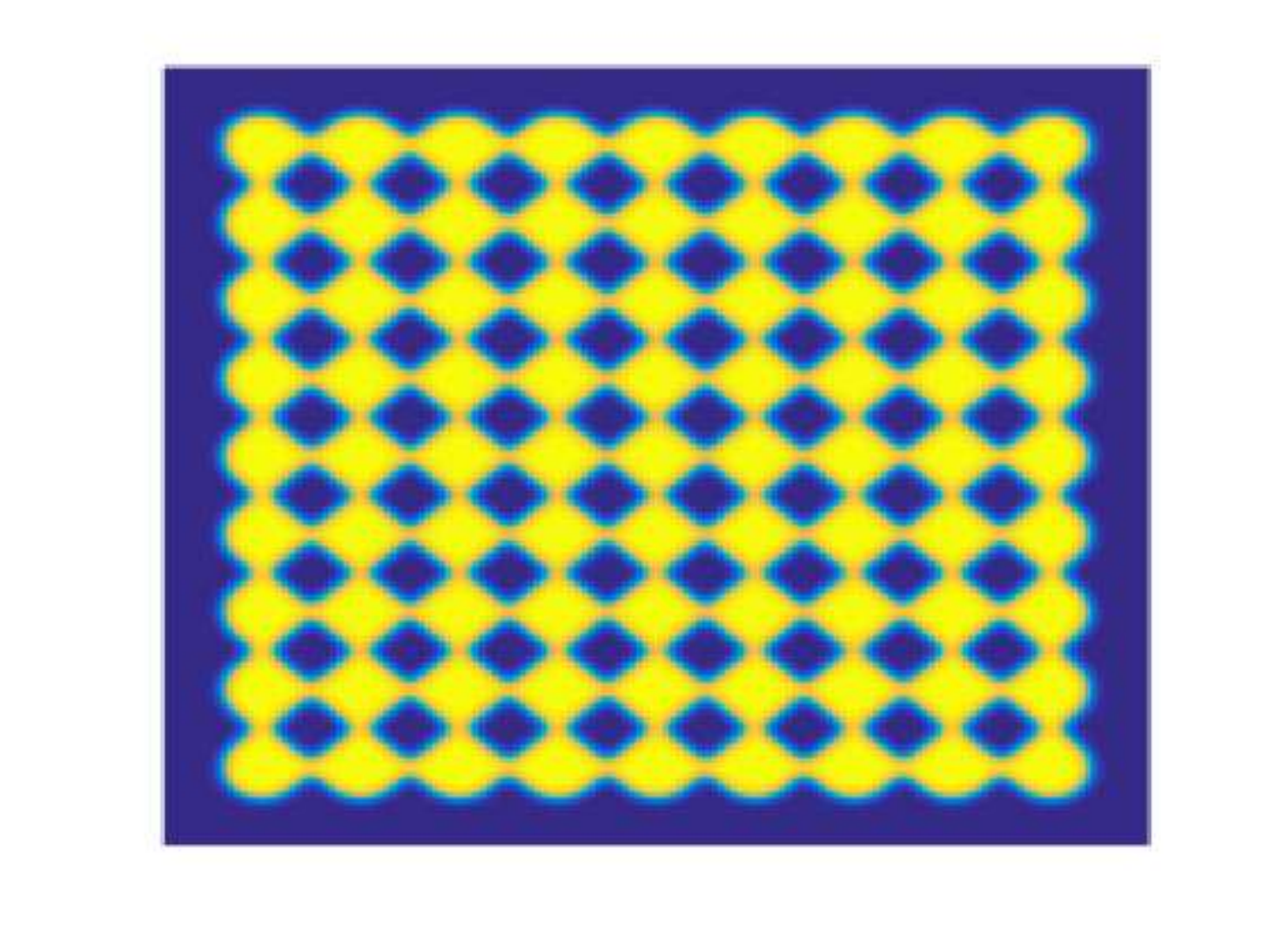}
}
\subfigure[t=1]
{
\includegraphics[width=3.8cm,height=3.8cm]{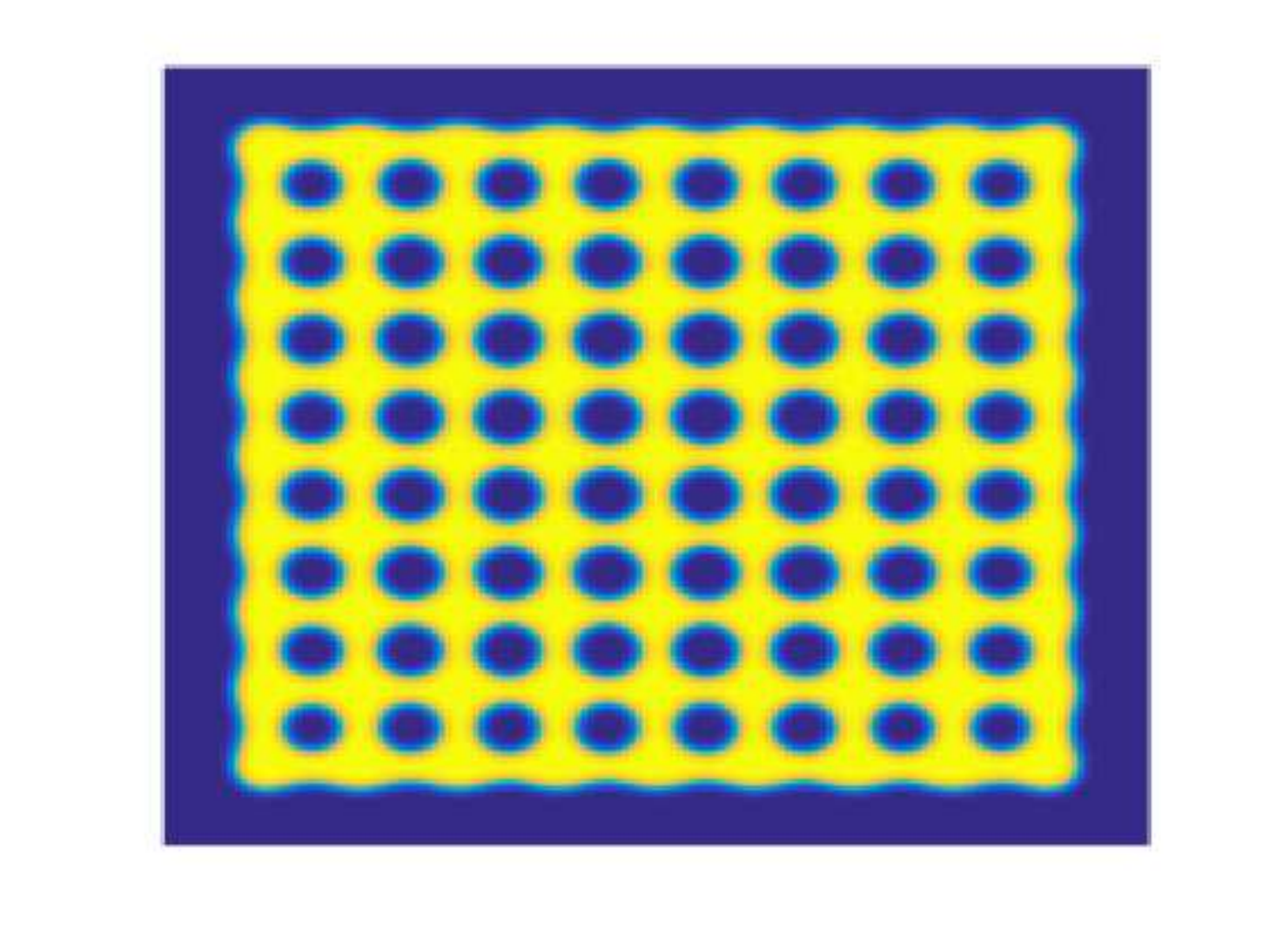}
}
\subfigure[t=3]
{
\includegraphics[width=3.8cm,height=3.8cm]{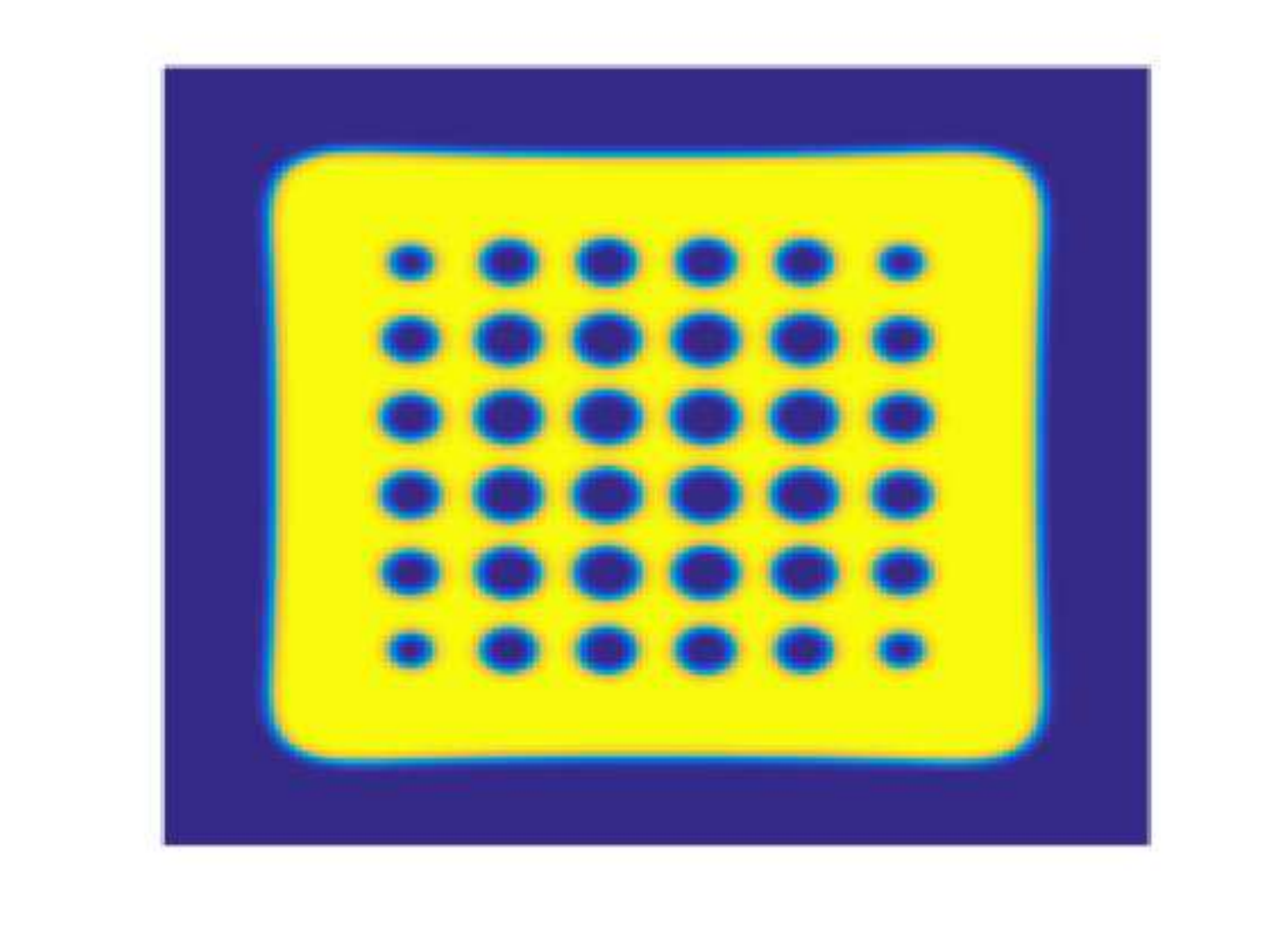}
}
\quad
\subfigure[t=4.2]
{
\includegraphics[width=3.8cm,height=3.8cm]{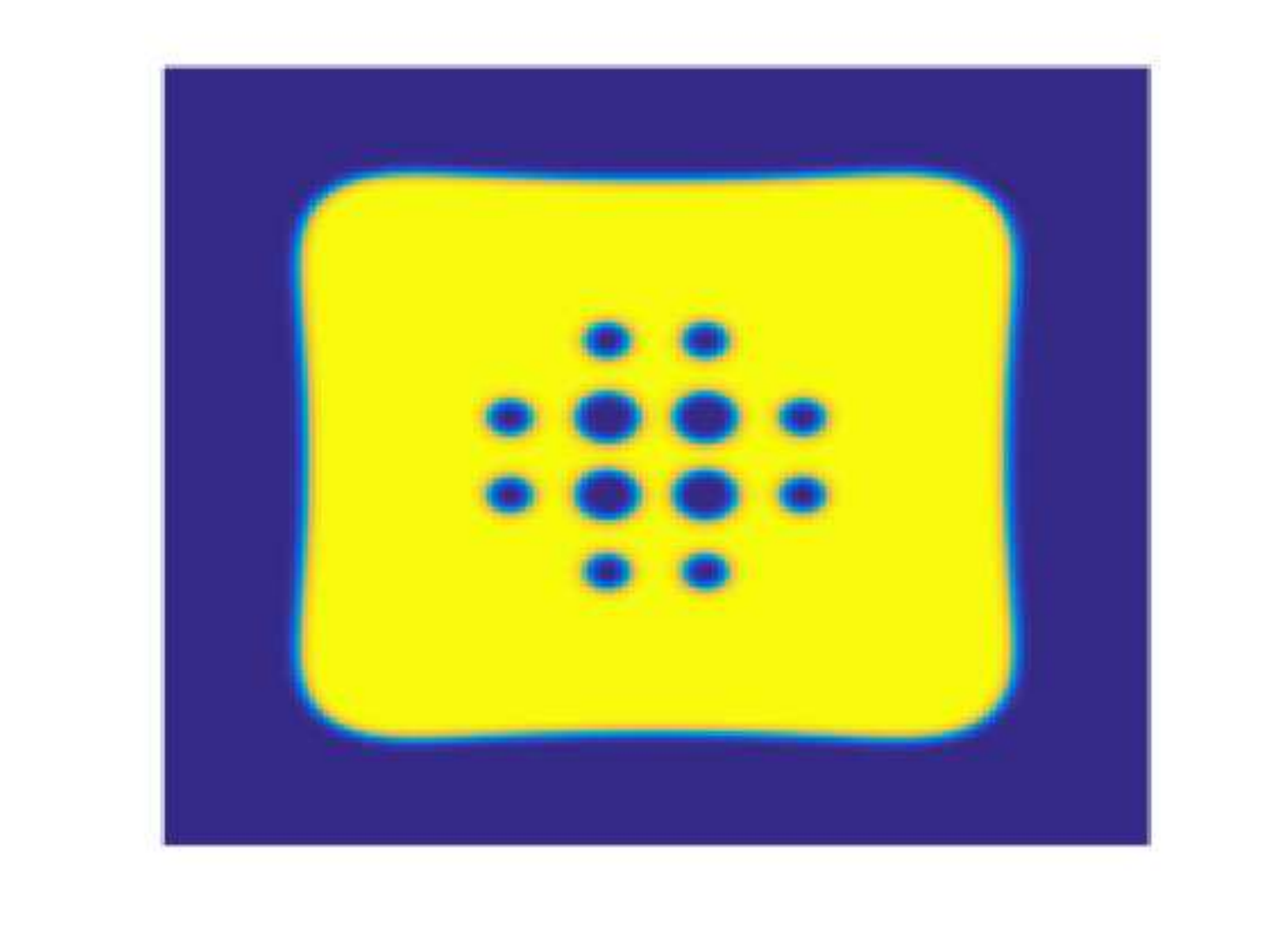}
}
\subfigure[t=4.8]
{
\includegraphics[width=3.8cm,height=3.8cm]{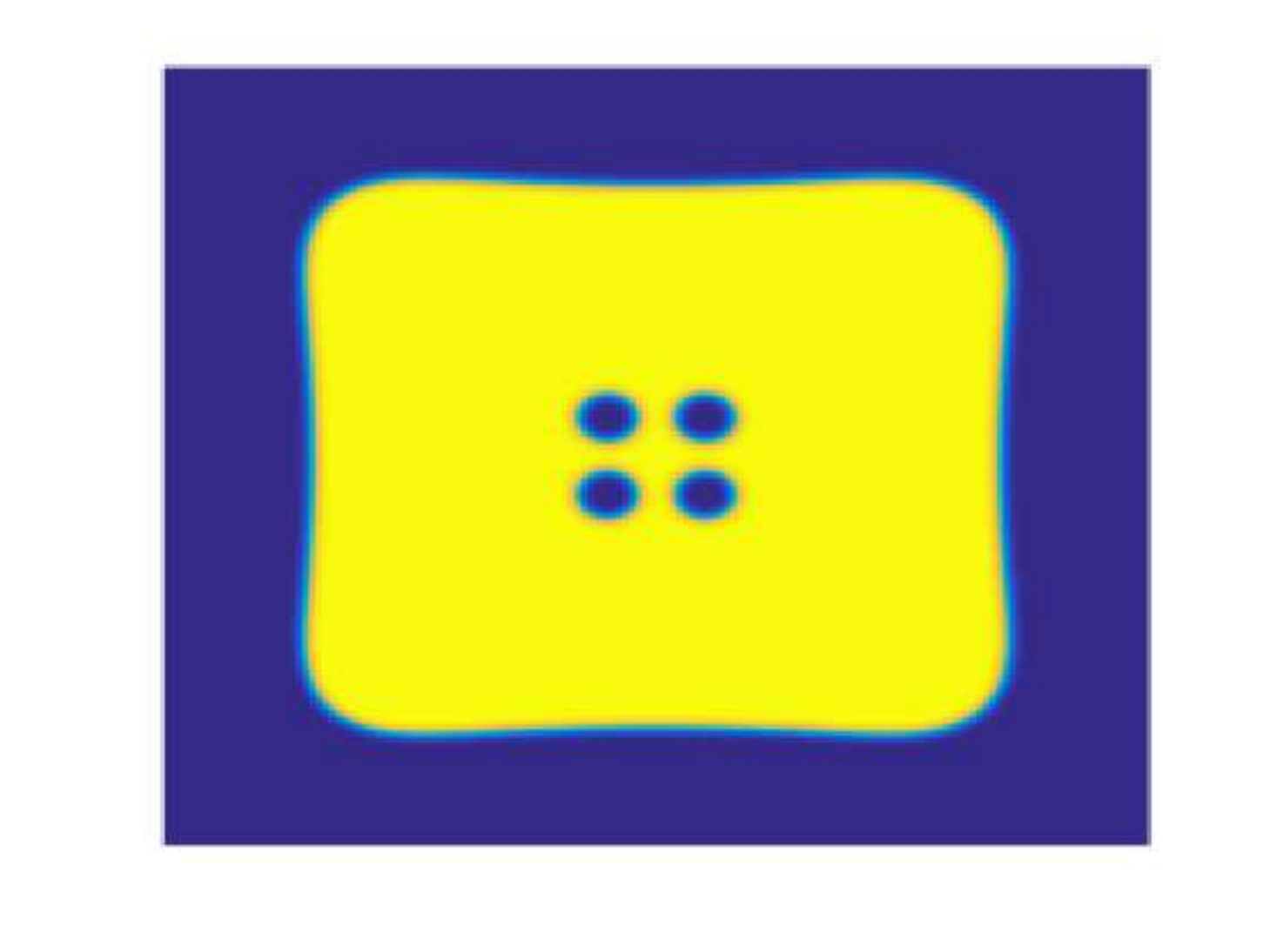}
}
\subfigure[t=10]
{
\includegraphics[width=3.8cm,height=3.8cm]{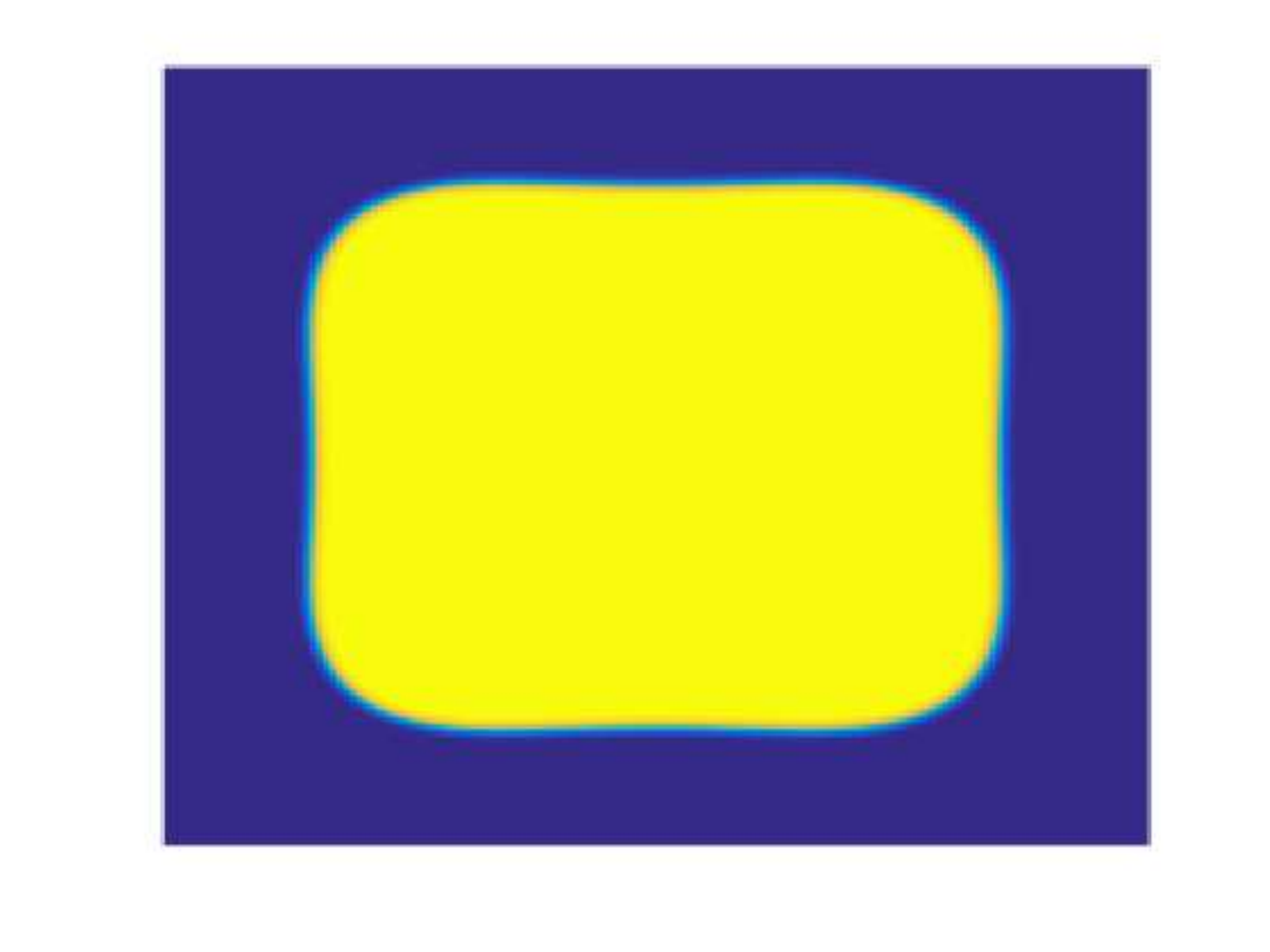}
}
\subfigure[t=100]
{
\includegraphics[width=3.8cm,height=3.8cm]{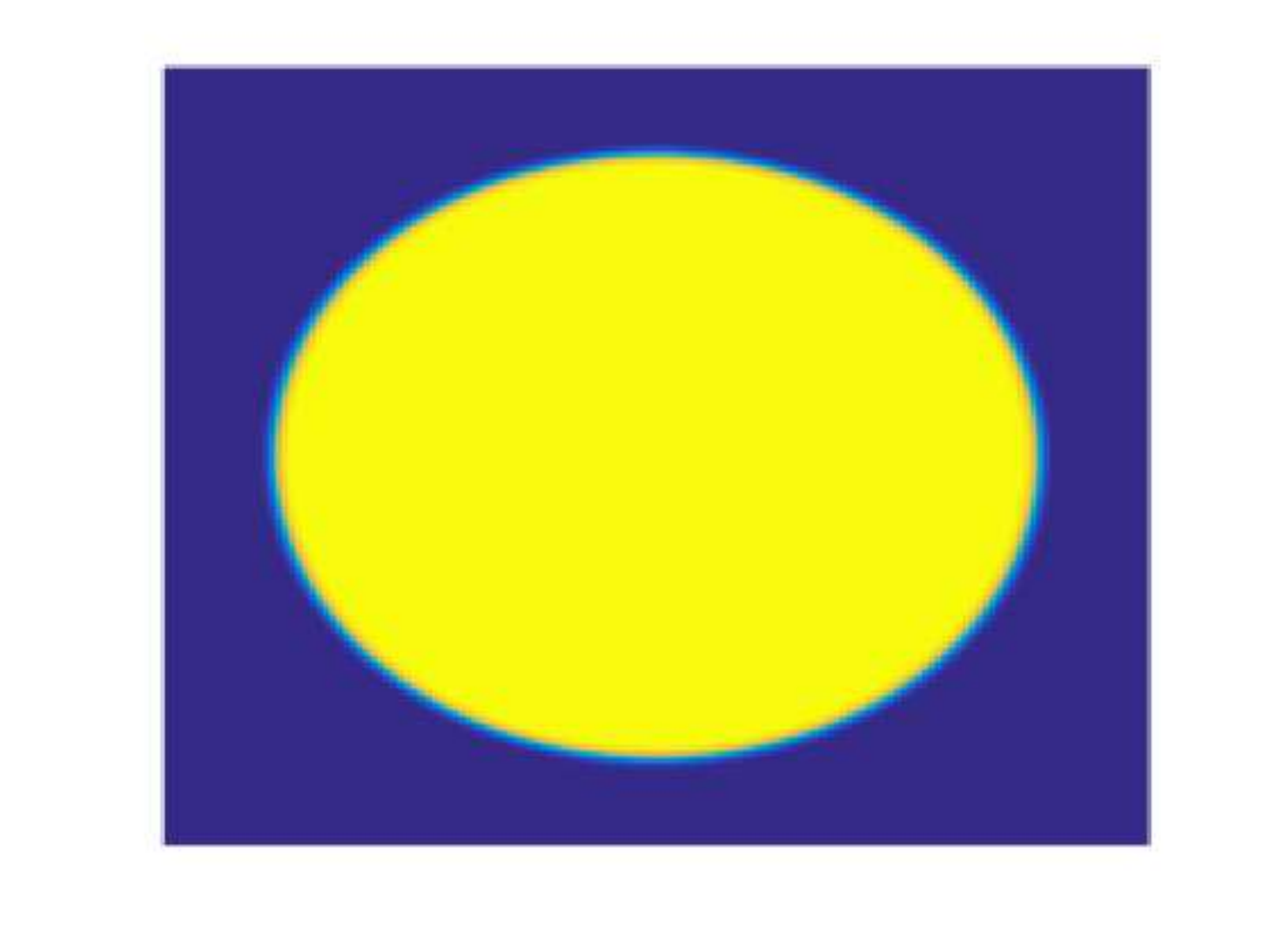}
}
\caption{Snapshots of the phase variable $\phi$ and phase interface are taken at t=0, 0.5, 1, 3, 4.2, 4.8, 10 and 100 with $\Delta t=0.01$ for example 2.}\label{fig:fig2}
\end{figure}
\begin{figure}[htp]
\centering
\includegraphics[width=8cm,height=8cm]{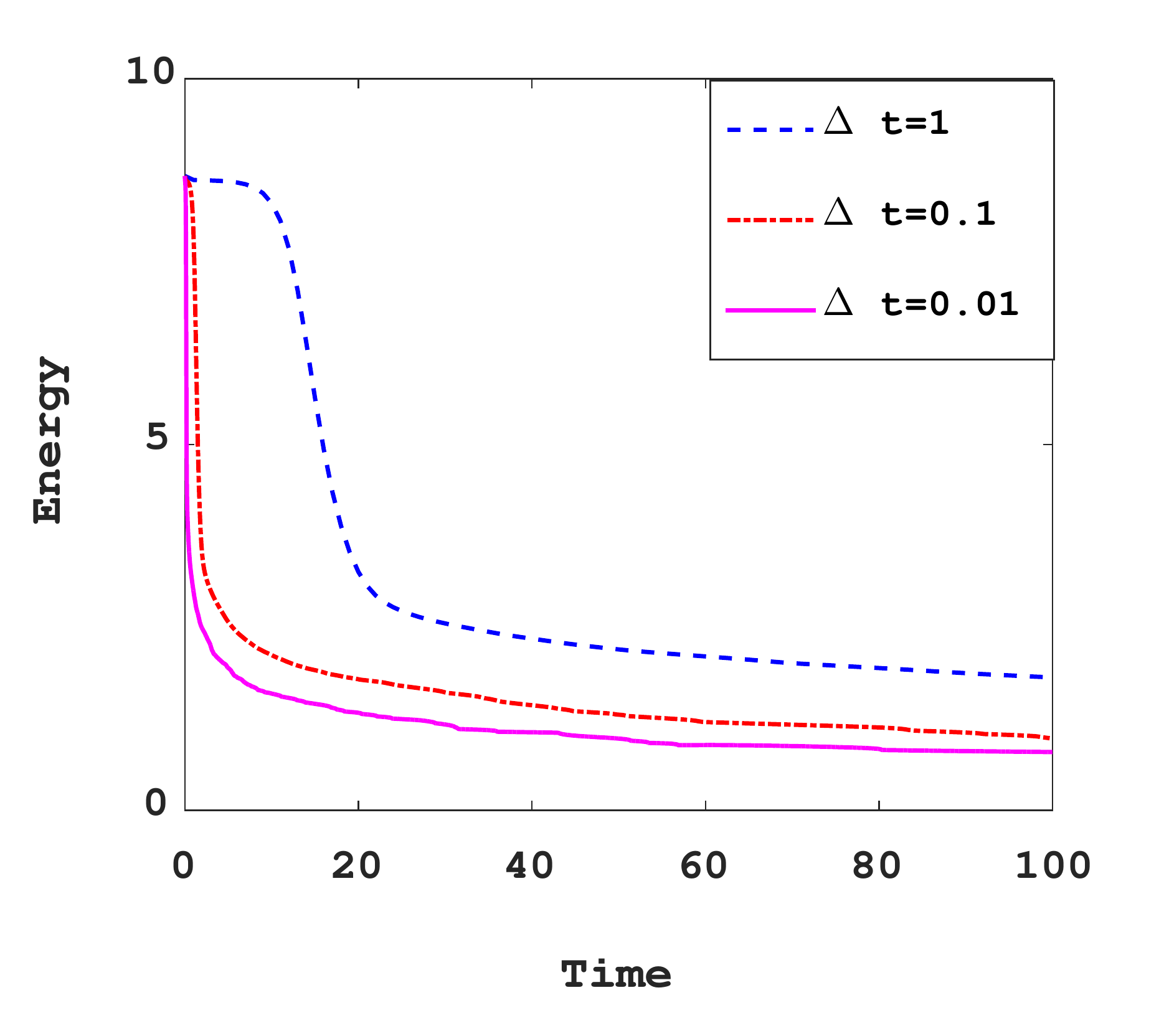}
\includegraphics[width=8cm,height=8cm]{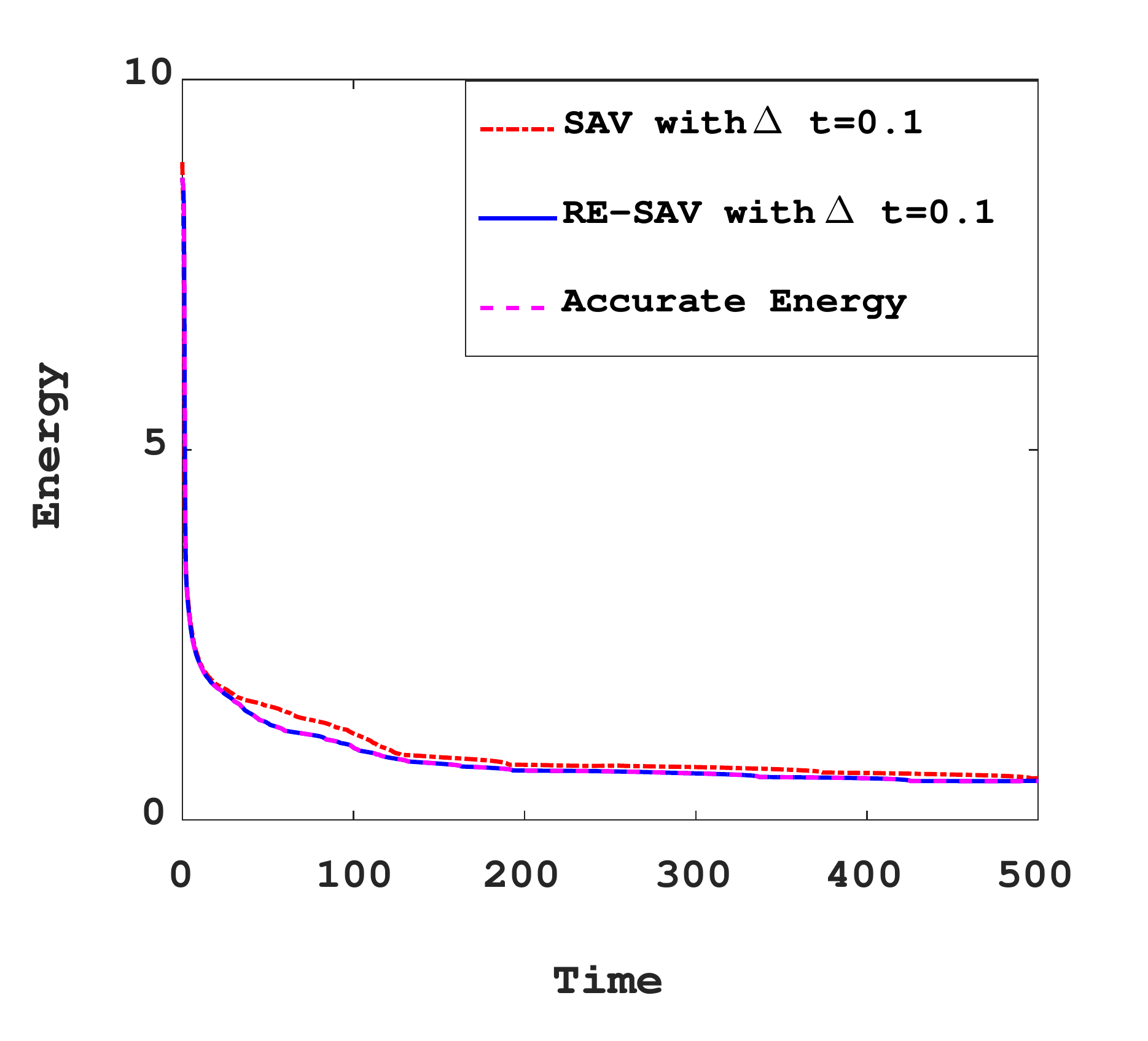}
\caption{Left: time evolution of the energy functional for three different time steps of $\Delta t=0.01$, $0.1$ and $1$. Right: energy evolution for the traditional SAV approach and the proposed RE-SAV approach with $\Delta t=0.1$.}\label{fig:fig3}
\end{figure}
\subsection{Swift-Hohenberg equations}
In this subsection, we will simulate the phase transition behavior of the Swift-Hohenberg equation with quadratic-cubic nonlinearity for the proposed RE-SAV approach. The similar numerical example can be found in many articles such as \cite{li2017efficient,yang2017linearly}. The Swift-Hohenberg model is a very important phase field crystal model which can be described many crystal phenomena such as edge dislocations \cite{BerryDiffusive}, deformation and plasticity in nanocrystalline materials \cite{StefanovicPhase}, fcc ordering \cite{WuPhase}, epitaxial growth and zone refinement \cite{ElderModeling}. Elder \cite{elder2002modeling} firstly proposed the phase field crystal (PFC) model based on density functional theory in 2002. This model can simulate the evolution of crystalline microstructure on atomistic length and diffusive time scales. It naturally incorporates elastic and plastic deformations and multiple crystal orientations, and can be applied to many different physical phenomena.

In particular, consider the following Swift-Hohenberg free energy:
\begin{equation*}
E(\phi)=\int_{\Omega}\left(\frac{1}{4}\phi^4-\frac{g}{3}\phi^3+\frac{1}{2}\phi\left(-\epsilon+(1+\Delta)^2\right)\phi\right)d\textbf{x},
\end{equation*}
where $\textbf{x} \in \Omega \subseteq \mathbb{R}^d$, $\phi$ is the density field, $g\geq0$ and $\epsilon>0$ are constants with physical significance, $\Delta$ is the Laplacian operator.

Considering a gradient flow in $H^{-1}$, one can obtain the Swift-Hohenberg equation under the constraint of mass conservation as follows:
\begin{equation*}
\frac{\partial \phi}{\partial t}=\Delta\mu=\Delta\left(\phi^3-g\phi^2-\epsilon\phi+(1+\Delta)^2\phi\right), \quad(\textbf{x},t)\in\Omega\times Q,
\end{equation*}
which is a sixth-order nonlinear parabolic equation and can be applied to simulate various phenomena such as crystal growth, material hardness and phase transition. Here $Q=(0,T]$, $\mu=\frac{\delta E}{\delta \phi}$ is called the chemical potential. In particular, the above Swift-Hohenberg will be the classical phase field crystal model when $g=0$.

In the following example, we simulate the benchmark simulation for the SH model.

\textbf{Example 3}: The initial condition is
\begin{equation}\label{section5_e2}
\aligned
&\phi_0(x,y)=0.07+0.07\times rand(x,y),
\endaligned
\end{equation}
where the $rand(x,y)$ is the random number in $[-1,1]$ with zero mean.  In this test, the domain is $\Omega=[0,100]^2$ with mesh size $256\times256$. The order parameter is $\epsilon=0.025$ and we consider two different $g=0$ and $g=1$ to test the effect on the crystallization.

We show the phase transition behavior of the density field for different values at various times in Figures \ref{fig:fig4-1} with $g=0$ and \ref{fig:fig4-2} with $g=1$. We observe that for different $g$, the shape and rate of crystallization of crystals are different. In all cases, the process of the phase transition is qualitative agreement of the density fields. Similar computation results for phase field crystal model can be found in many articles such as in \cite{yang2017linearly}.

\begin{figure}[htp]
\centering
\subfigure[t=10]{
\includegraphics[width=5cm,height=5cm]{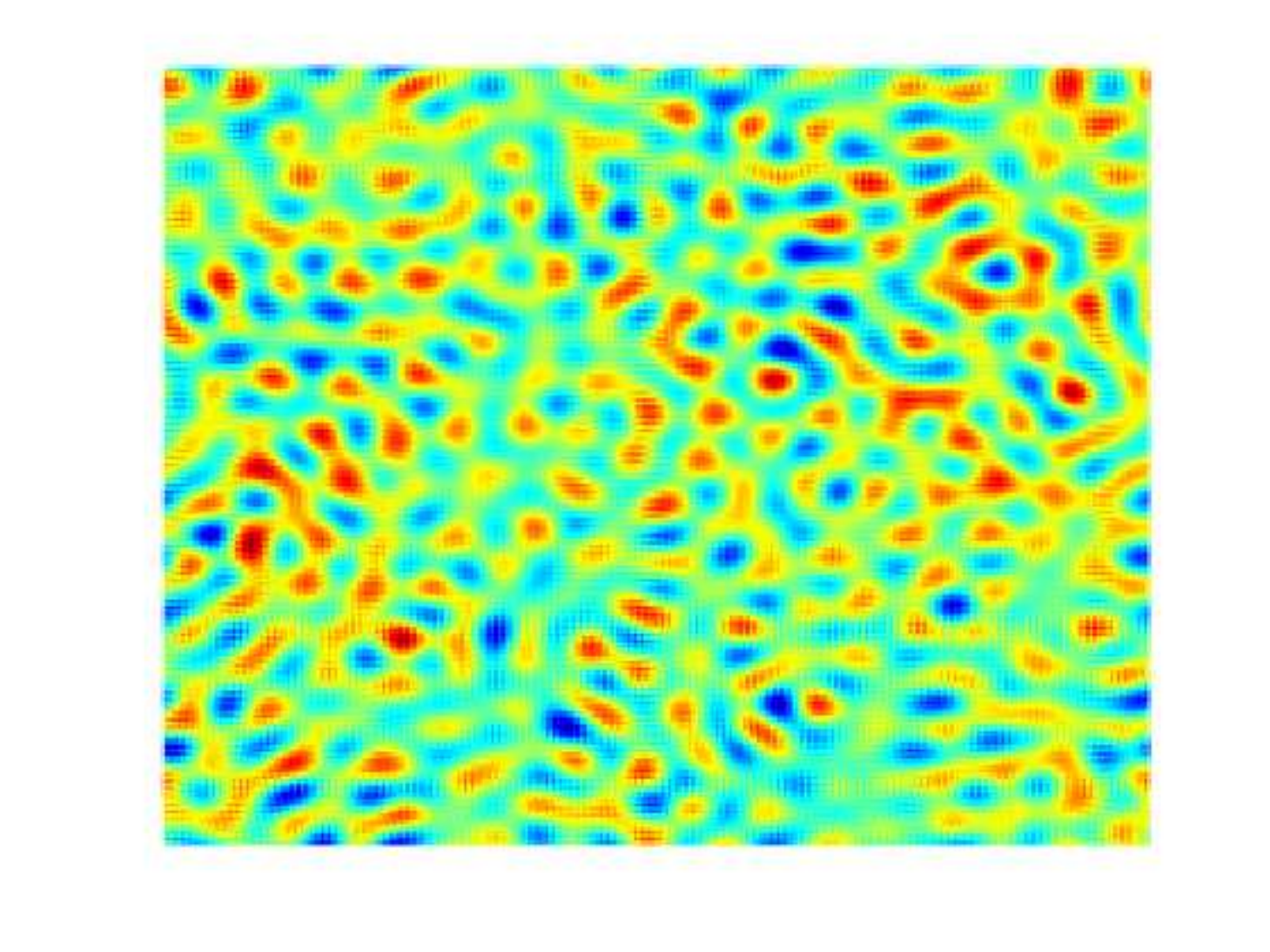}
}
\subfigure[t=100]
{
\includegraphics[width=5cm,height=5cm]{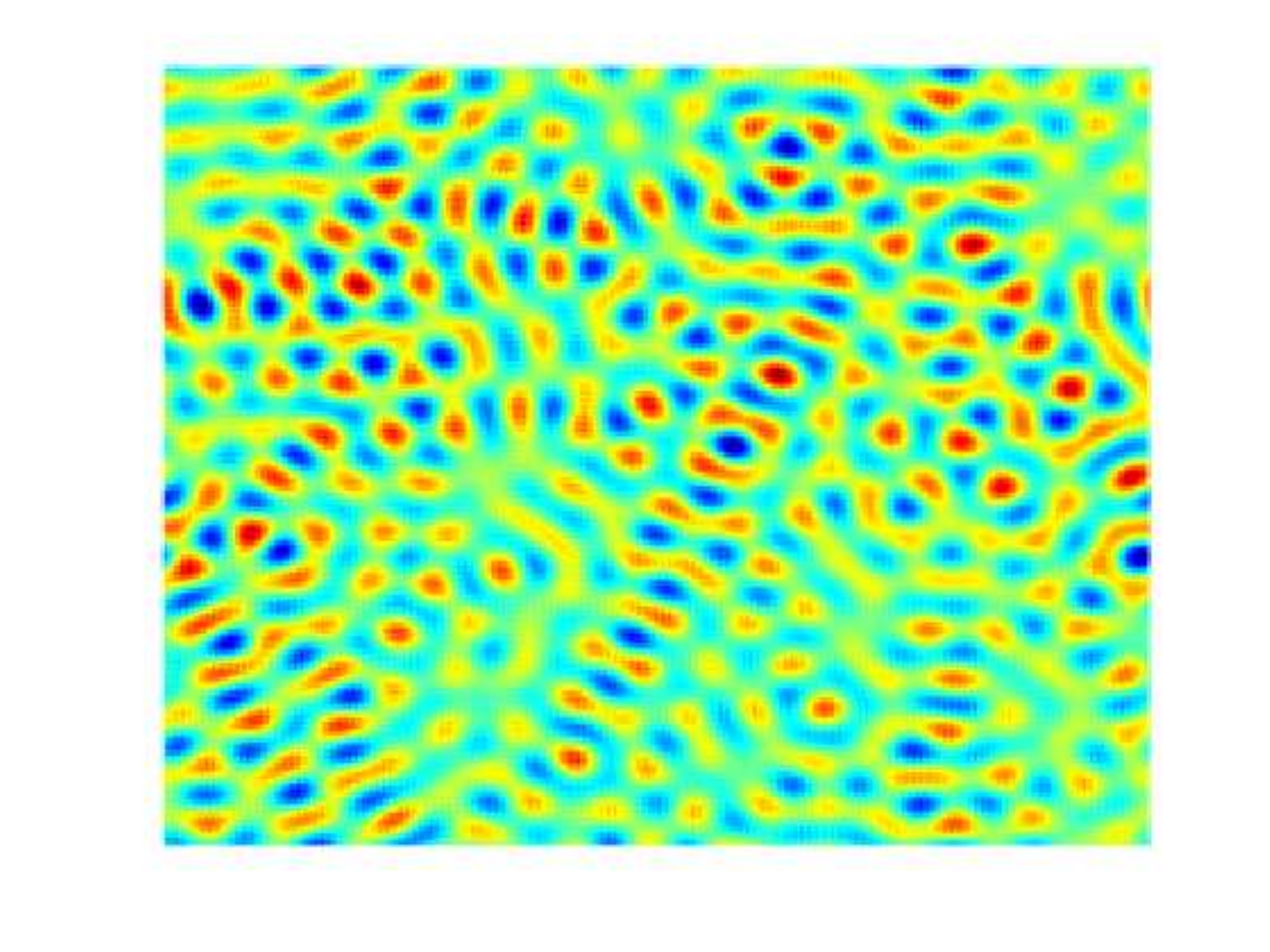}
}
\subfigure[t=300]
{
\includegraphics[width=5cm,height=5cm]{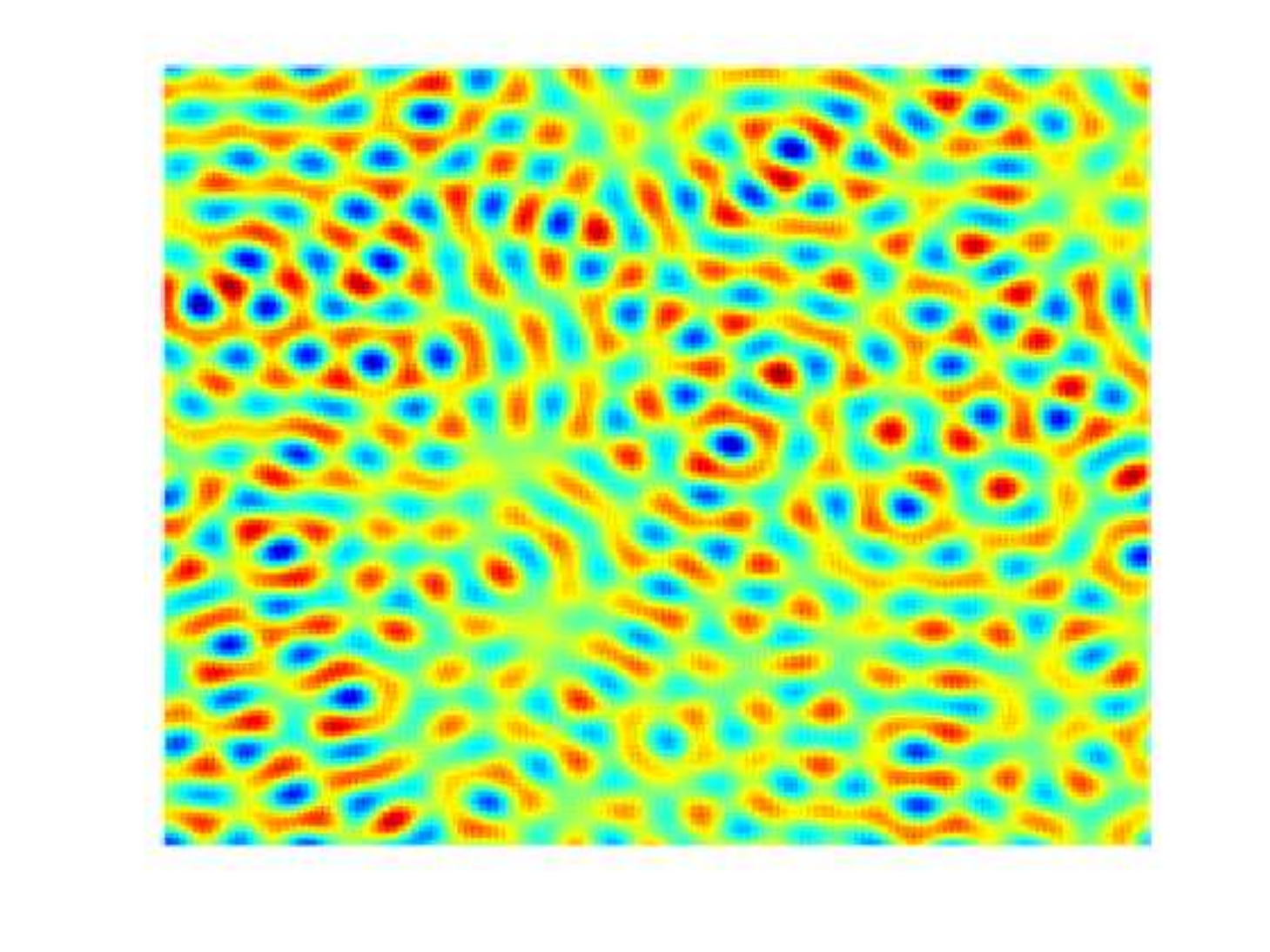}
}
\quad
\subfigure[t=500]{
\includegraphics[width=5cm,height=5cm]{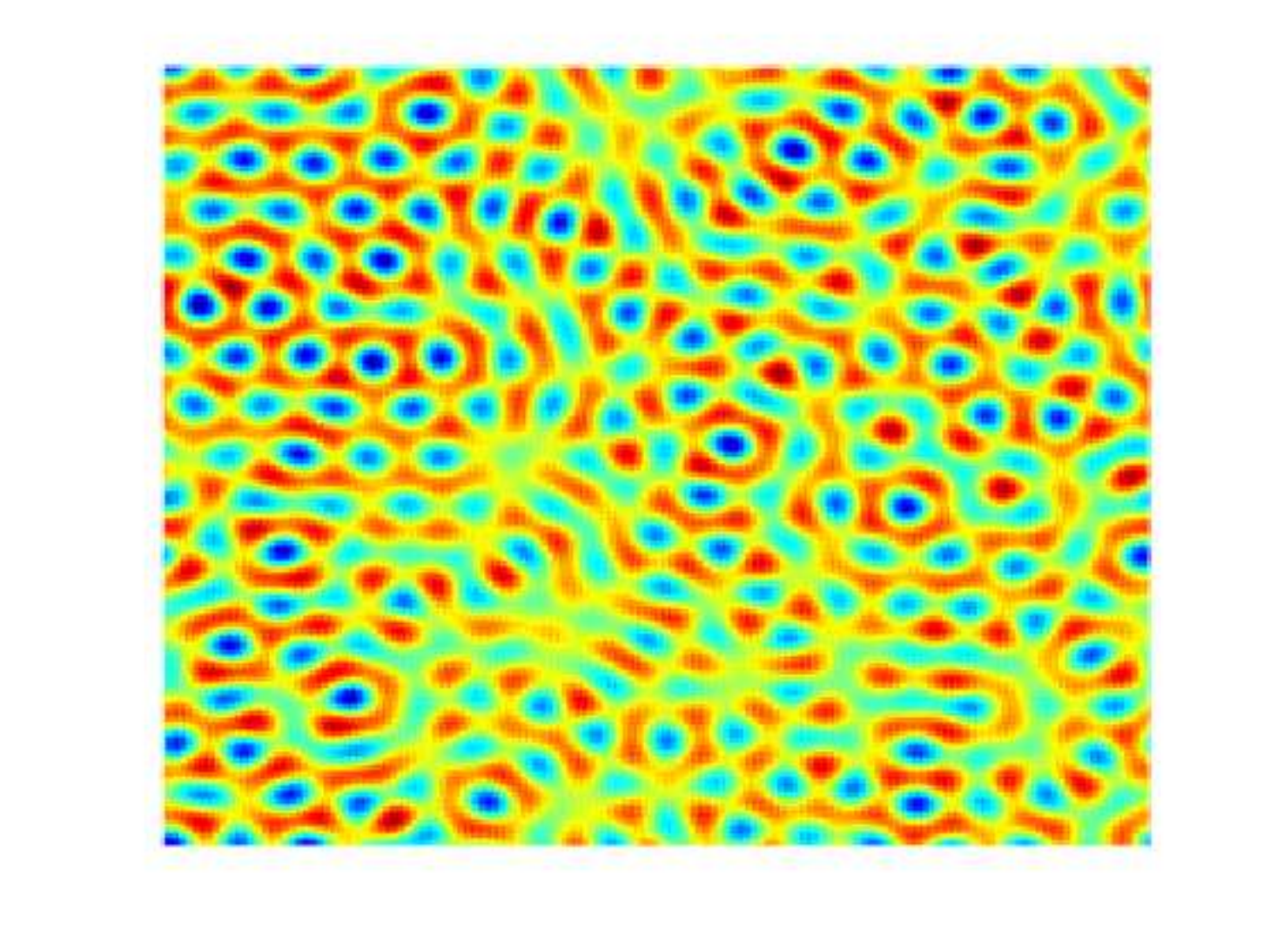}
}
\subfigure[t=800]
{
\includegraphics[width=5cm,height=5cm]{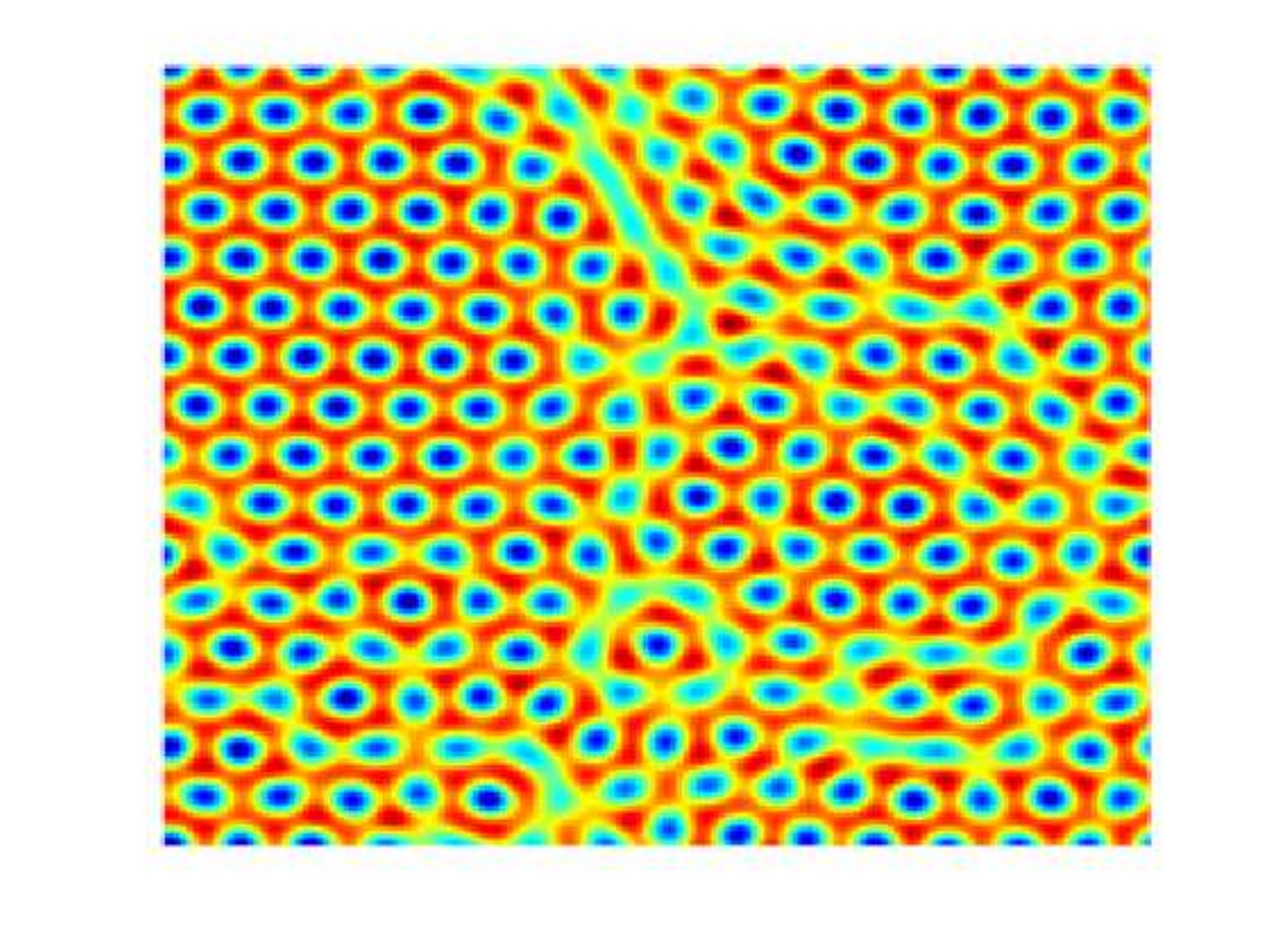}
}
\subfigure[t=1000]
{
\includegraphics[width=5cm,height=5cm]{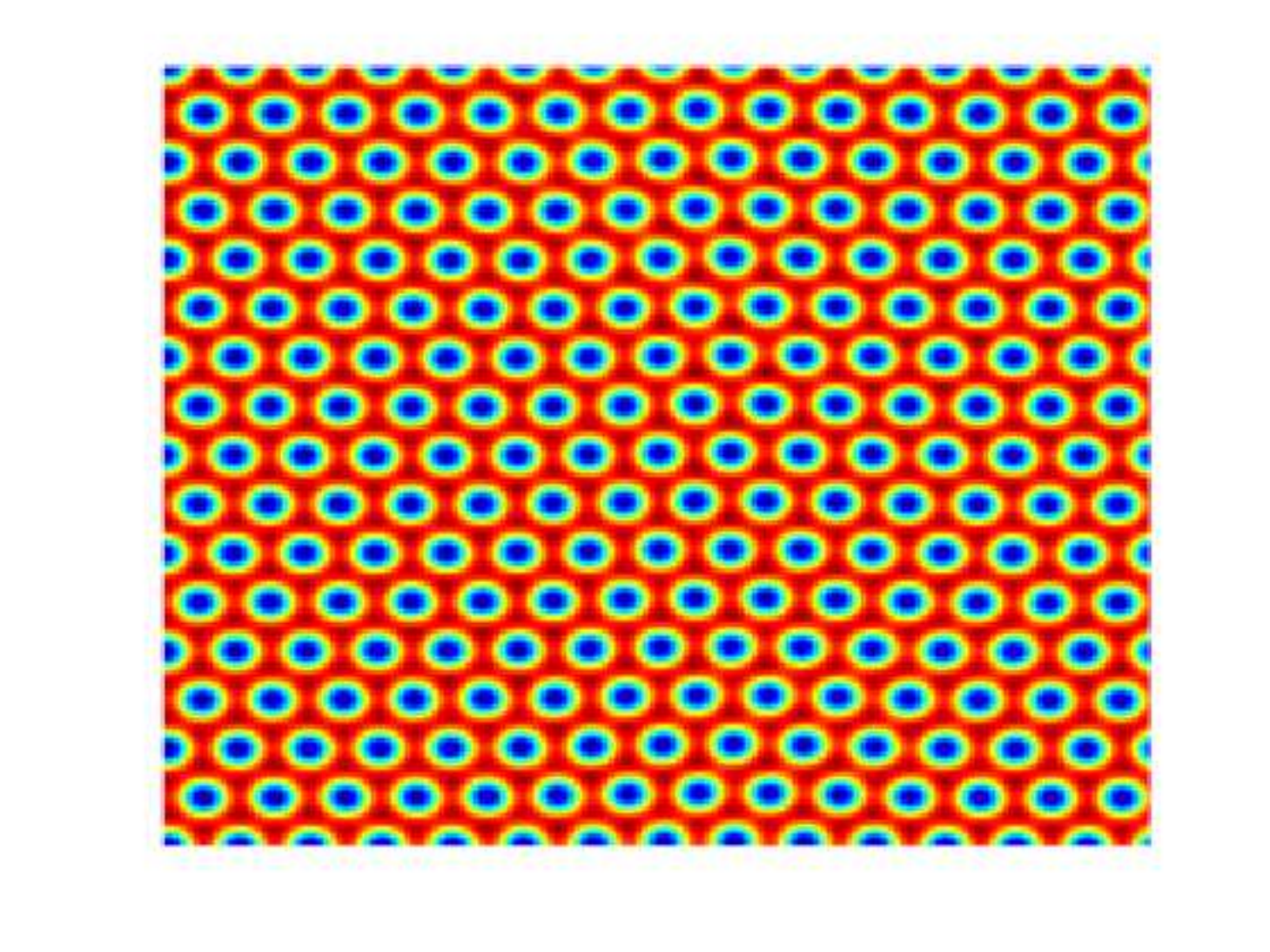}
}
\caption{Snapshots of the phase variable $\phi$ are taken at t=10, 100, 300, 500, 800, 1000 for example 3 with $g=0$.}\label{fig:fig4-1}
\end{figure}
\begin{figure}[htp]
\centering
\subfigure[t=1]{
\includegraphics[width=5cm,height=5cm]{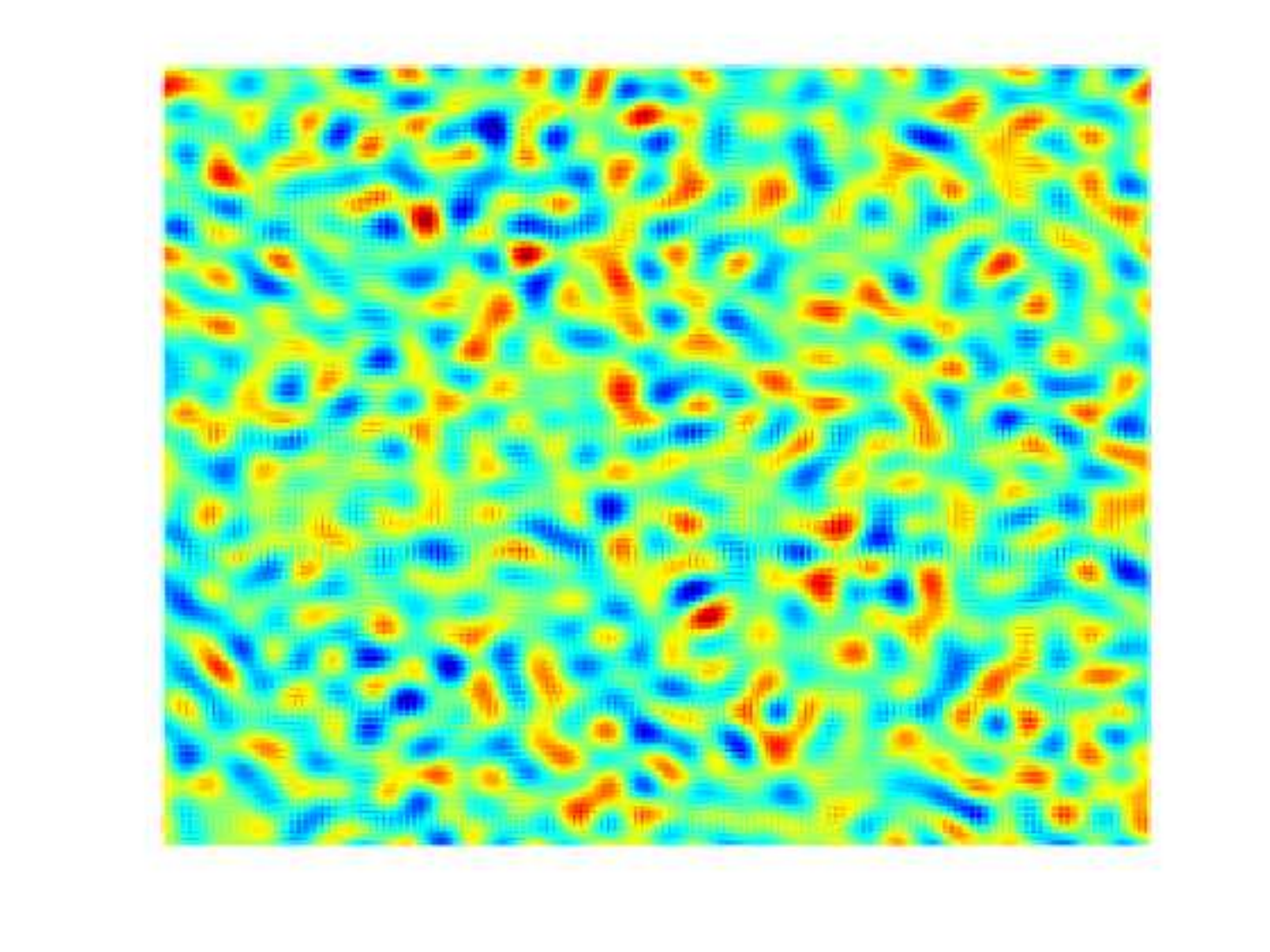}
}
\subfigure[t=10]
{
\includegraphics[width=5cm,height=5cm]{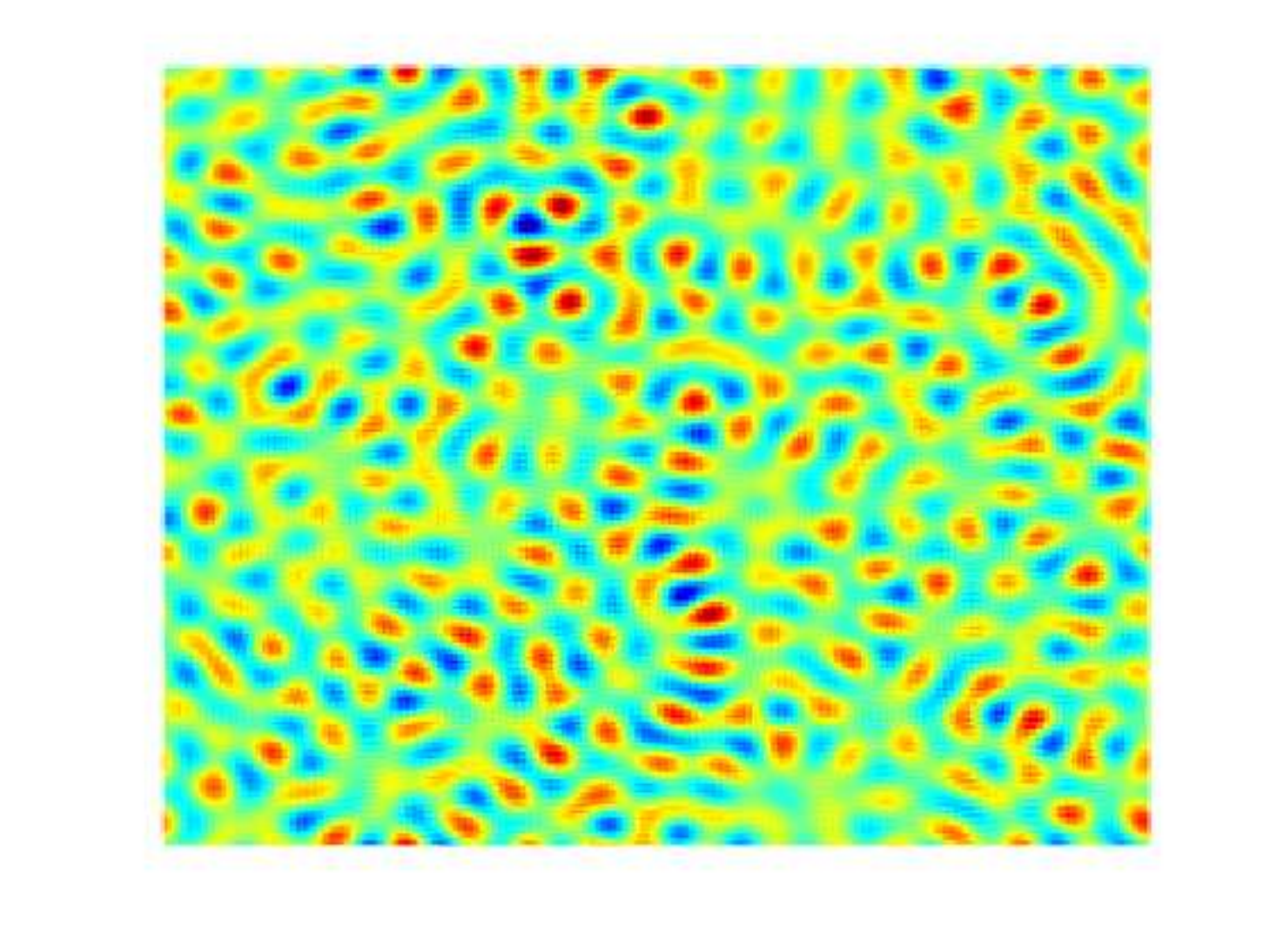}
}
\subfigure[t=20]
{
\includegraphics[width=5cm,height=5cm]{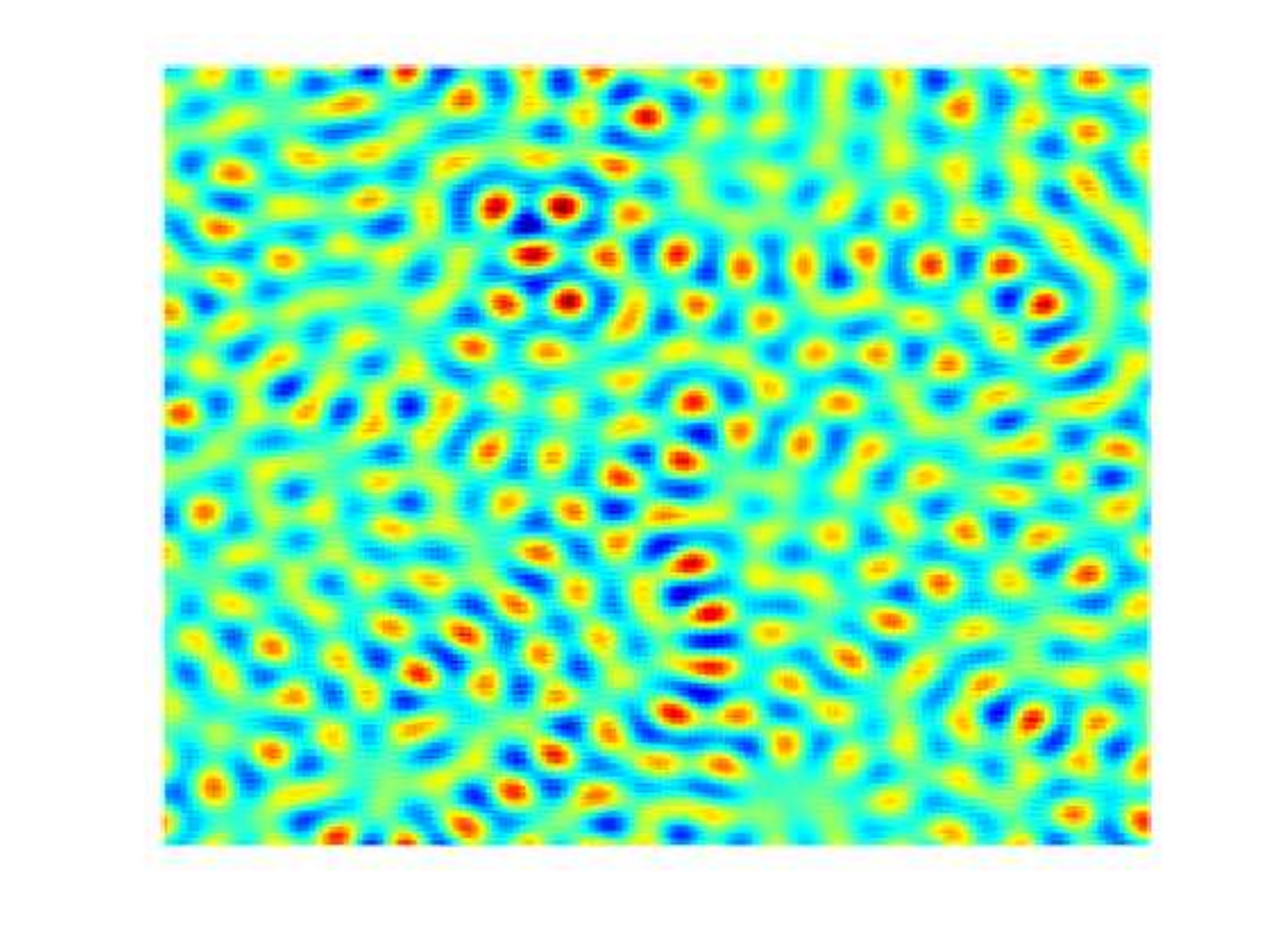}
}
\quad
\subfigure[t=30]{
\includegraphics[width=5cm,height=5cm]{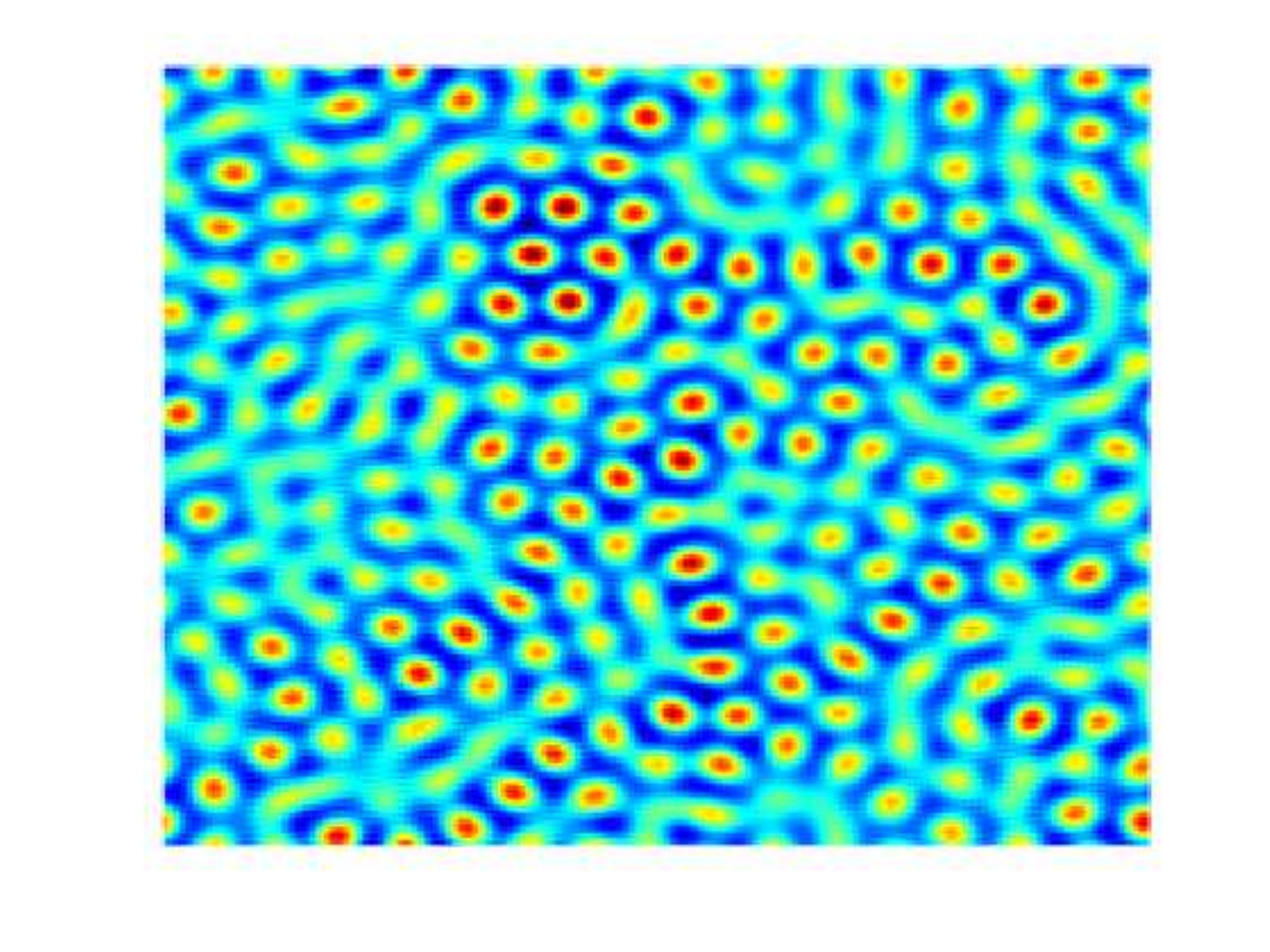}
}
\subfigure[t=40]
{
\includegraphics[width=5cm,height=5cm]{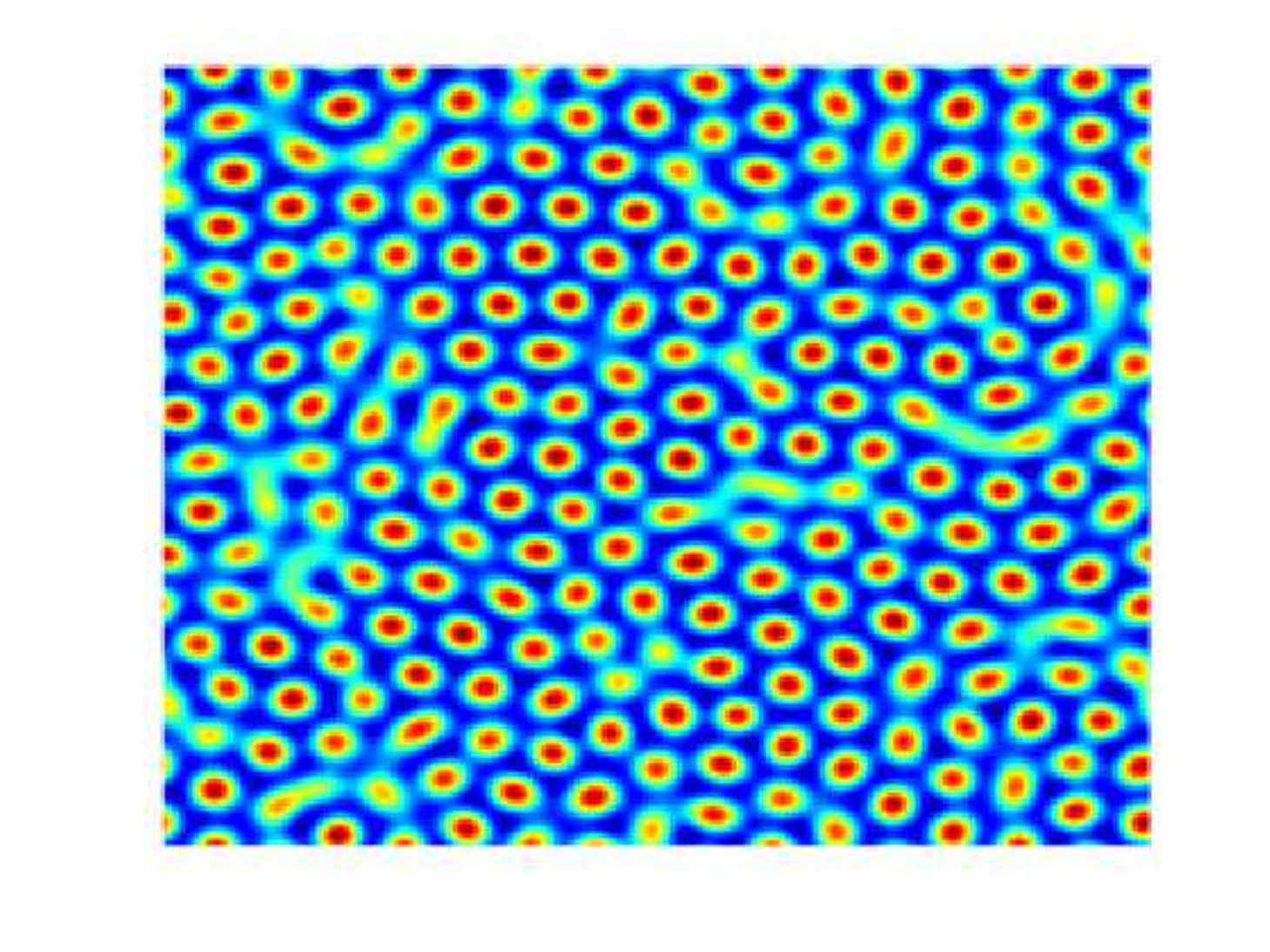}
}
\subfigure[t=100]
{
\includegraphics[width=5cm,height=5cm]{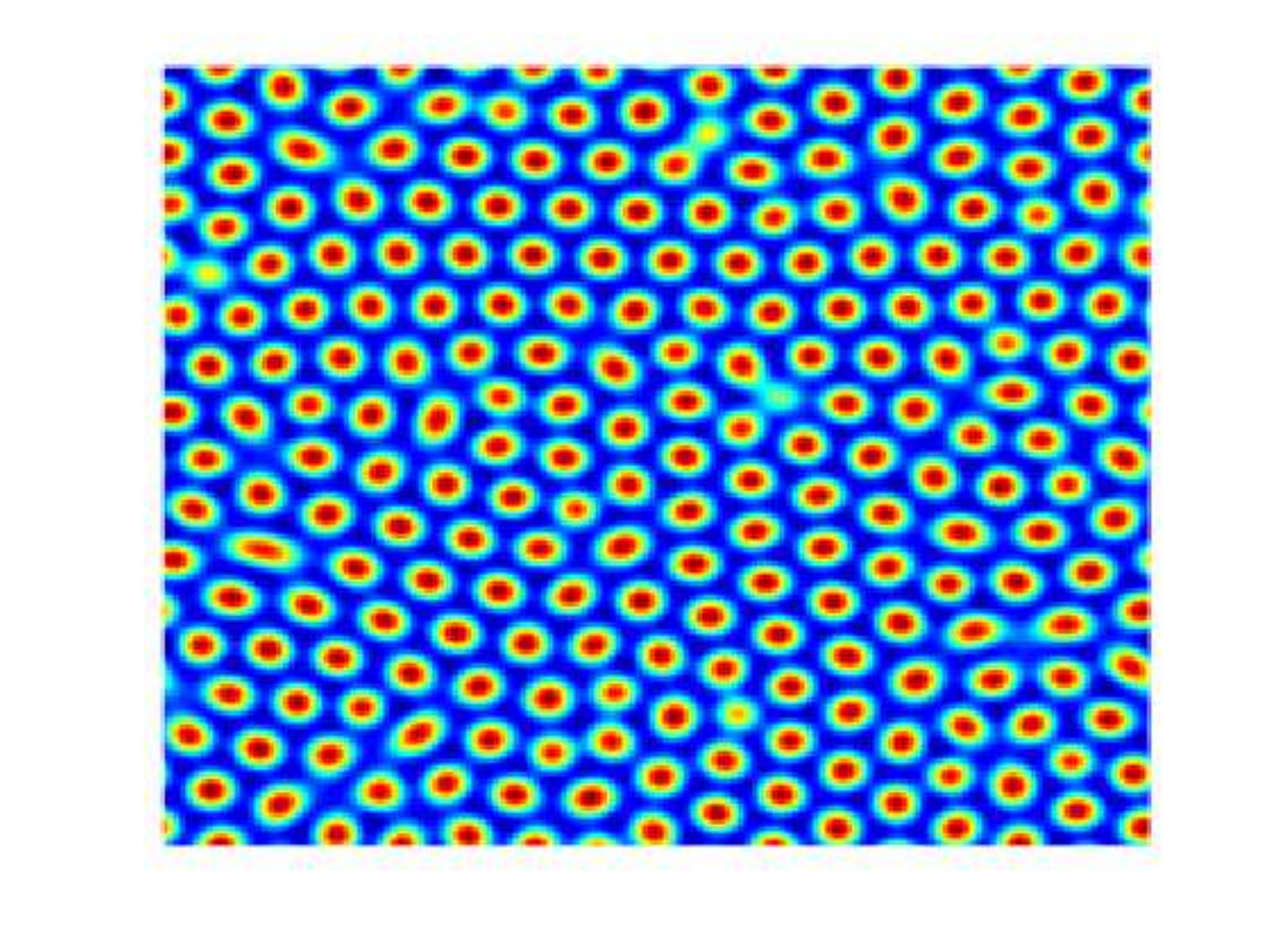}
}
\caption{Snapshots of the phase variable $\phi$ are taken at t=1, 10, 20, 30, 40, 100 for example 3 with $g=1$.}\label{fig:fig4-2}
\end{figure}

\textbf{Example 4}: The process of crystallization in a supercool liquid is very classical example. So in the following, we take $g=0$ and $\epsilon=0.25$ to start our simulation on a domain $[-200,200]\times[-200,200]$. We generated the three crystallites using random perturbations on three small square pathes. The following expression will be used to define the crystallites such as in \cite{yang2017linearly}:
\begin{equation*}
\phi(x_l,y_l)=\overline{\phi}+C\left(\cos(\frac{q}{\sqrt{3}}y_l)\cos(qx_l)-\frac12\cos(\frac{2q}{\sqrt{3}}y_l)\right),
\end{equation*}
where $x_l$, $y_l$ define a local system of cartesian coordinates that is oriented with the
crystallite lattice. The parameters $\overline{\phi}=0.285$, $C=0.446$ and $q=0.66$. The local cartesian system is defined as
\begin{equation*}
\aligned
&x_l(x,y)=xsin\theta+ycos\theta,\\
&y_l(x,y)=-xcos\theta+ysin\theta.
\endaligned
\end{equation*}
we set $512^2$ Fourier modes to discretize the two dimensional space. The centers of three pathes are located at $(150,150)$, $(250,300)$ and $(300,200)$ with $\theta=\pi/4$, $0$ and $-\pi/4$. The length of each square is 40. Figure \ref{fig:fig5} shows the snapshots of the density field $\phi$ at different times. We observe the growth of the crystalline phase. Three different crystal grains grow and become large enough to form grain boundaries finally. We plot the energy dissipative curve in Figure \ref{fig:fig6} using three time steps of $\Delta t=0.1$ and $1$. One can observe that the original energies decrease at all time steps.
\begin{figure}[htp]
\centering
\subfigure[t=0]{
\includegraphics[width=5cm,height=5cm]{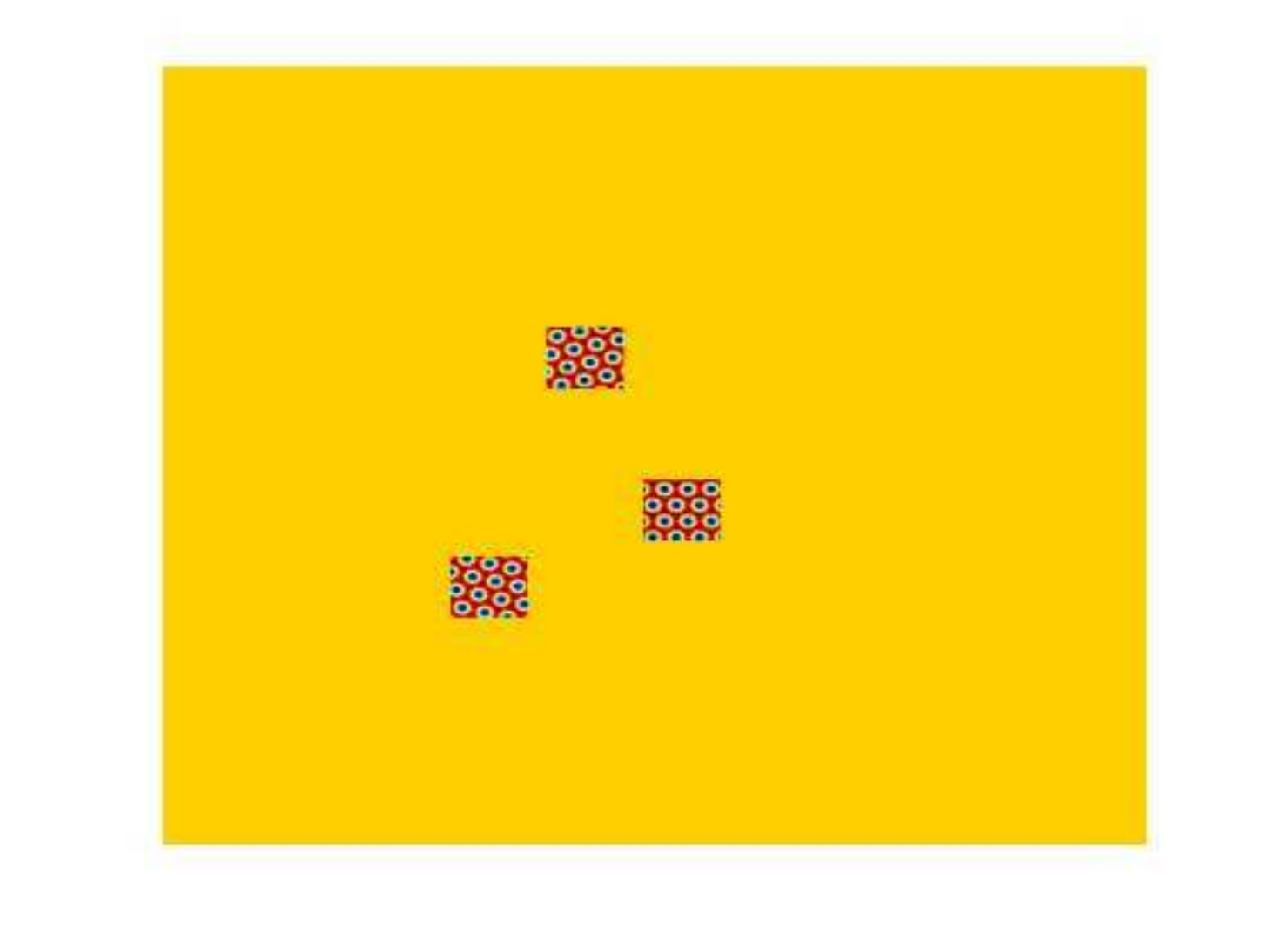}
}
\subfigure[t=50]
{
\includegraphics[width=5cm,height=5cm]{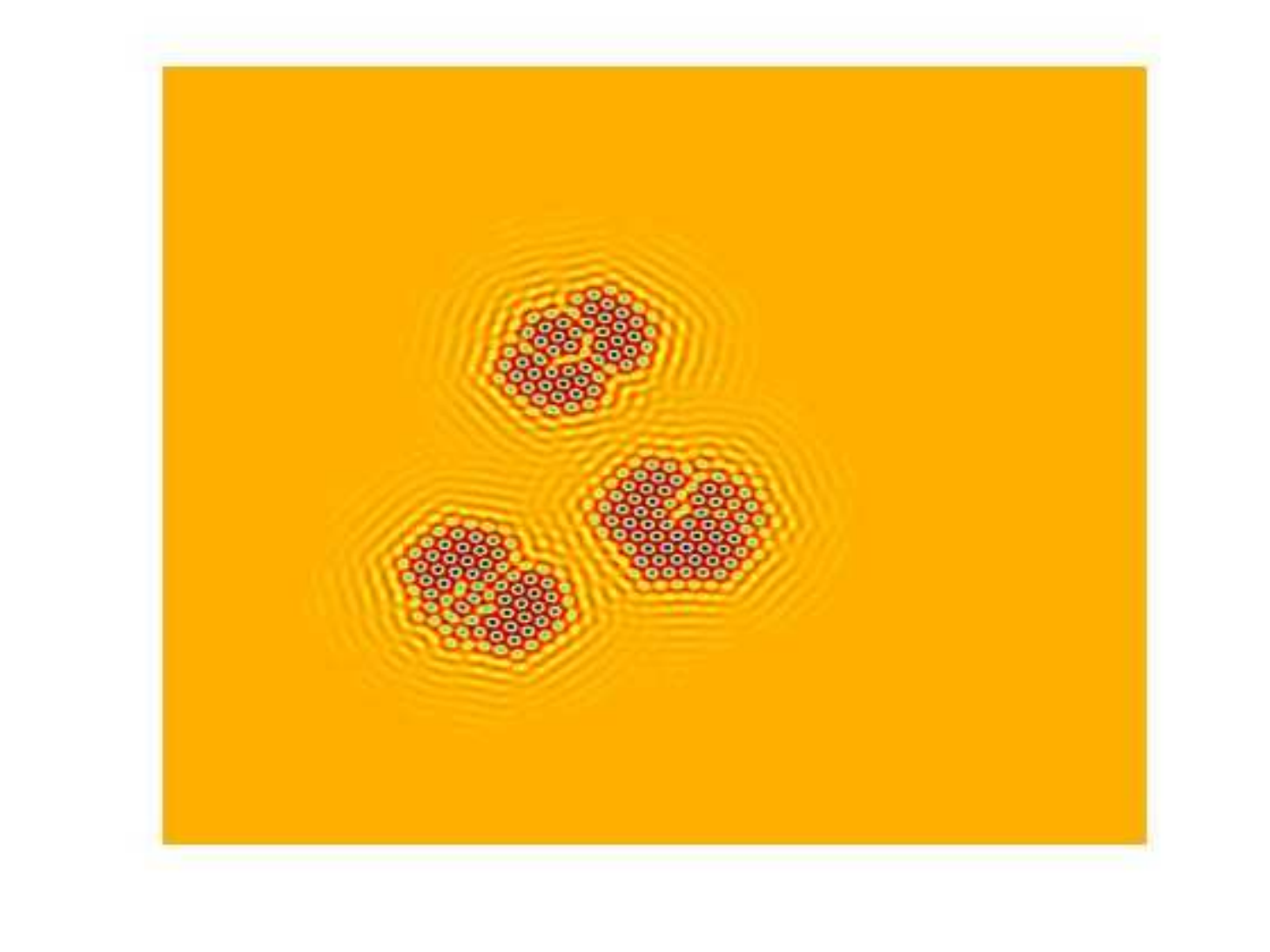}
}
\subfigure[t=100]
{
\includegraphics[width=5cm,height=5cm]{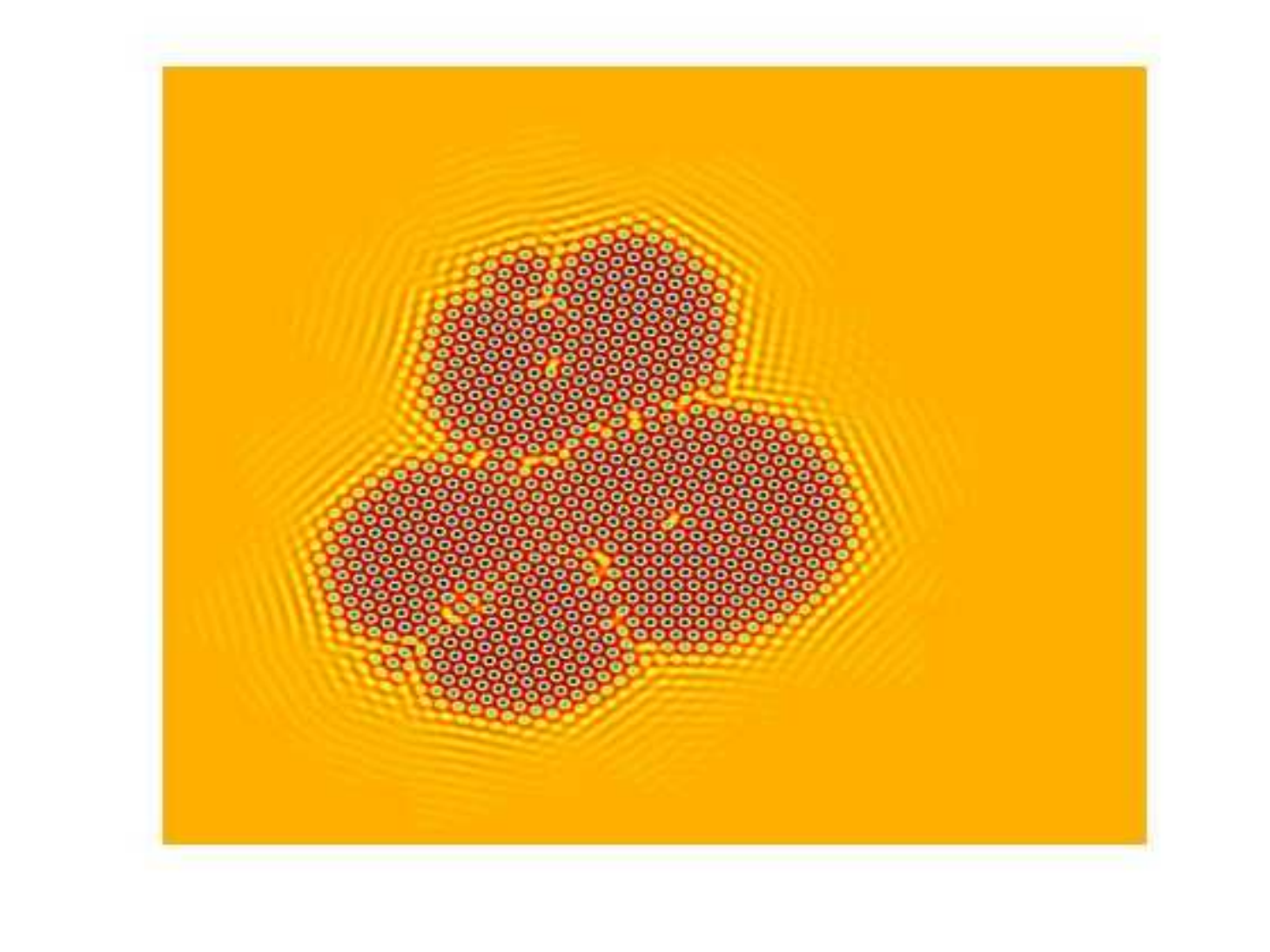}
}
\quad
\subfigure[t=500]{
\includegraphics[width=5cm,height=5cm]{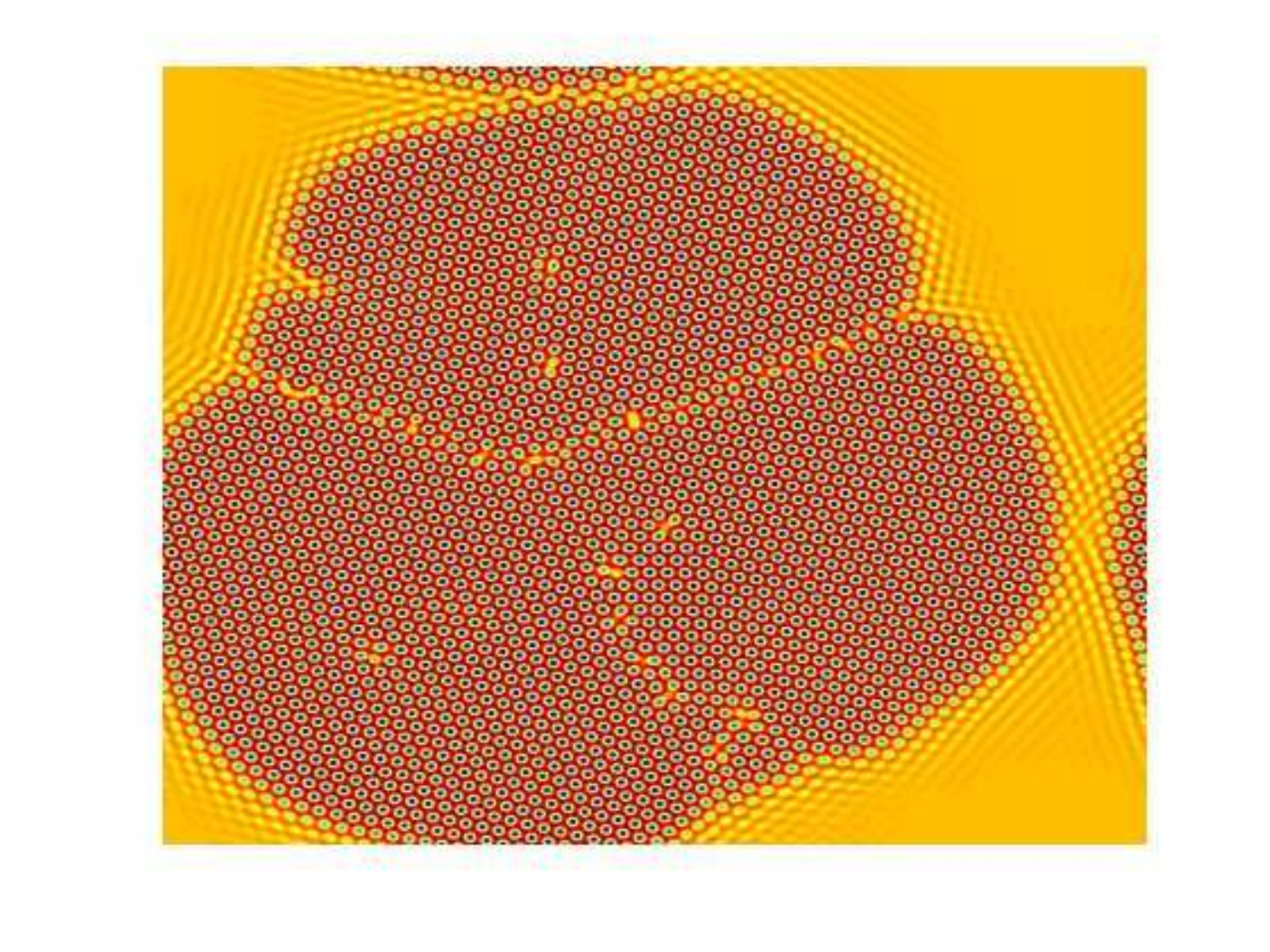}
}
\subfigure[t=600]
{
\includegraphics[width=5cm,height=5cm]{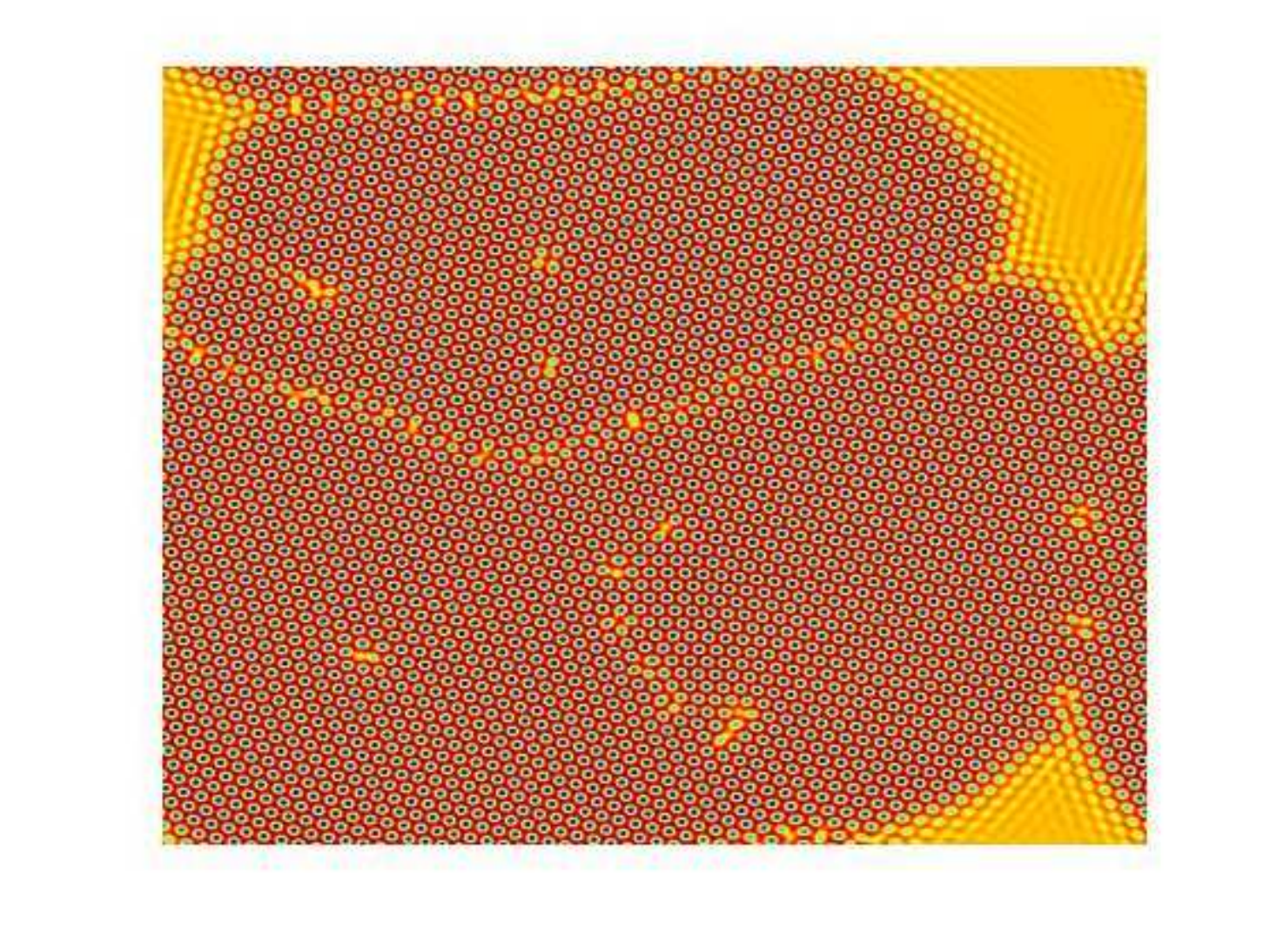}
}
\subfigure[t=800]
{
\includegraphics[width=5cm,height=5cm]{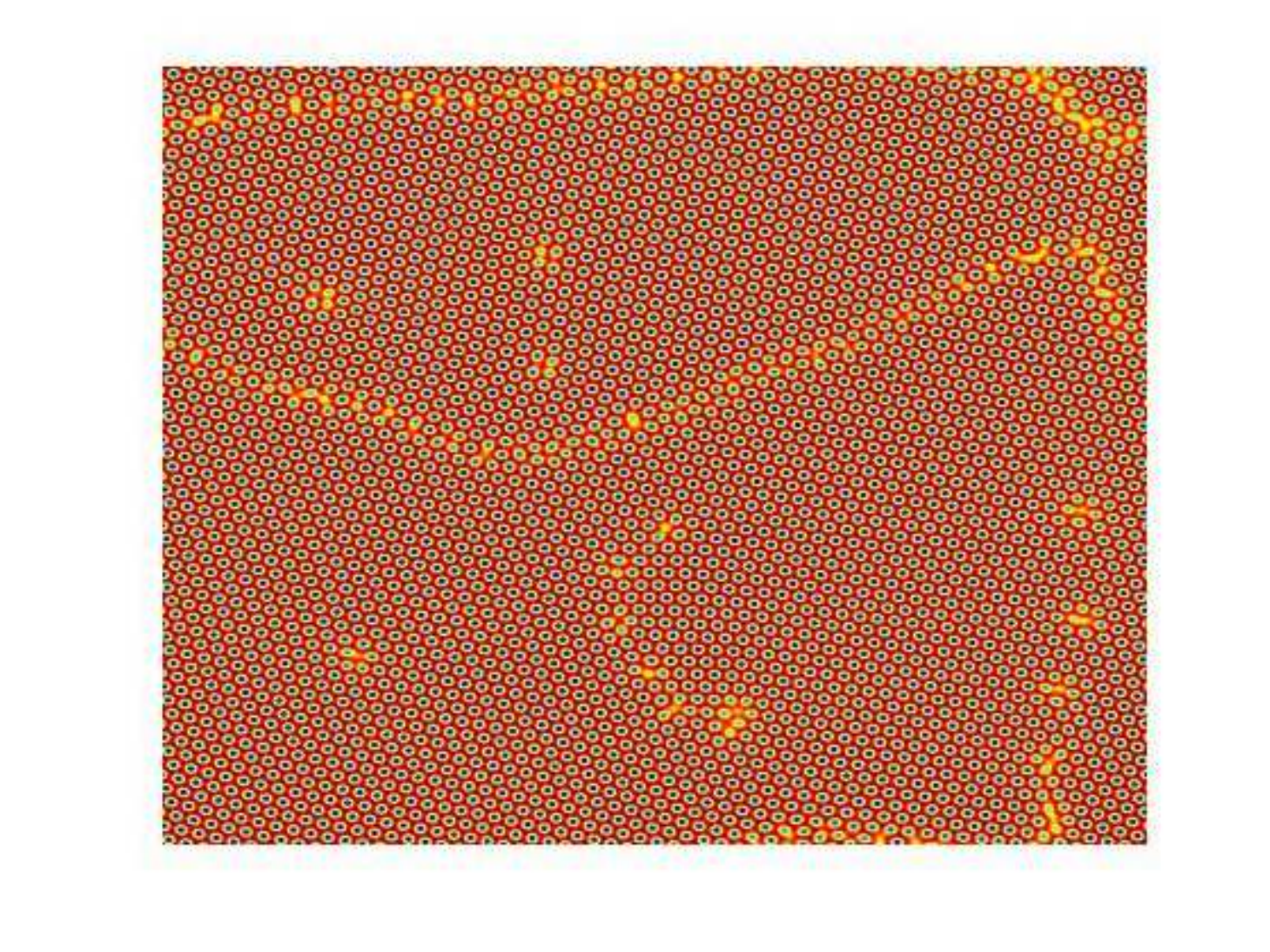}
}
\caption{Snapshots of the phase variable $\phi$ are taken at t=0, 200, 250, 350, 400, 500 for example 4.}\label{fig:fig5}
\end{figure}
\begin{figure}[htp]
\centering
\includegraphics[width=10cm,height=7cm]{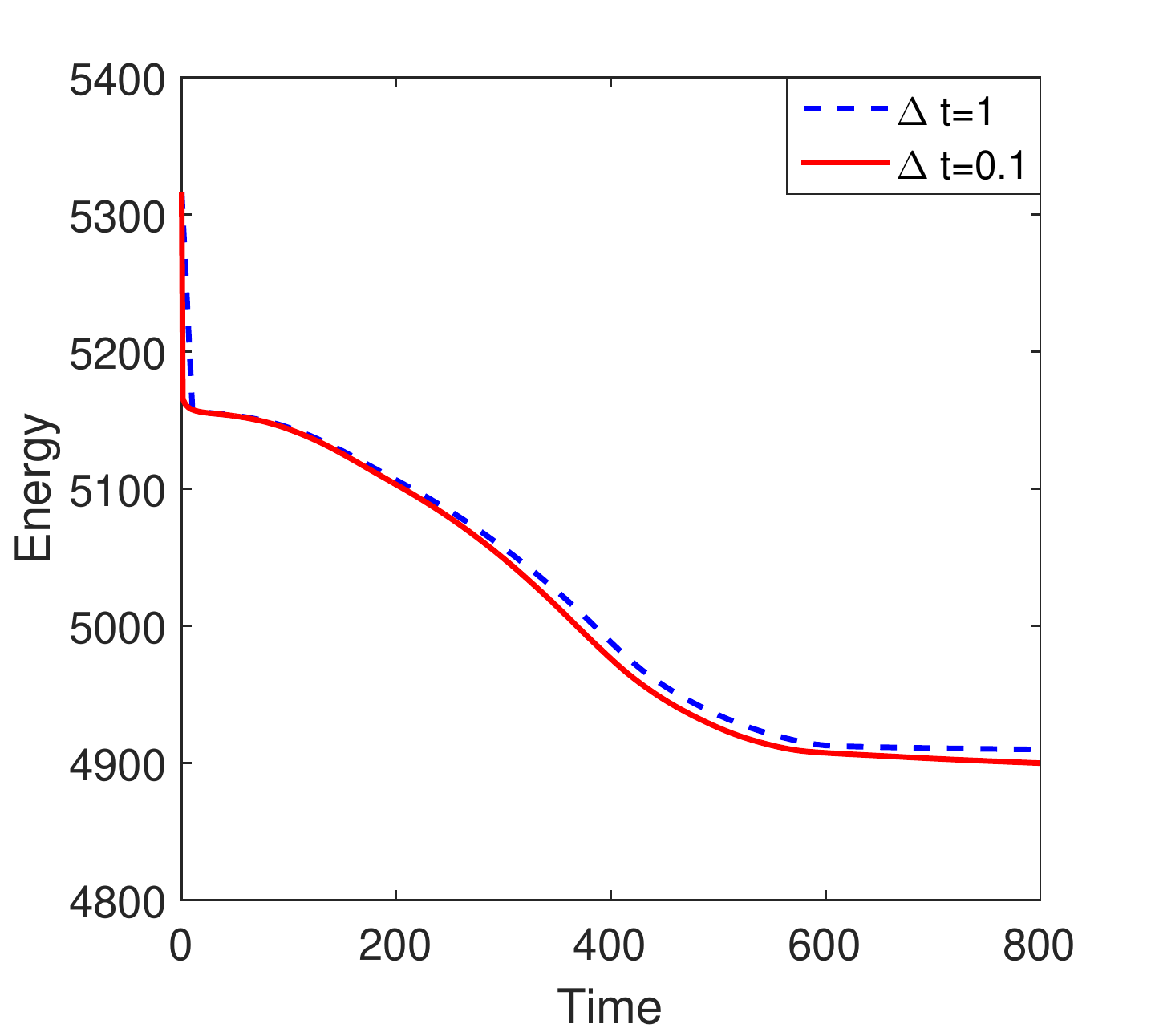}
\caption{Energy evolution of PFC model for example 4 with $\Delta t=0.1$ and $1$.}\label{fig:fig6}
\end{figure}
\section*{Acknowledgement}
No potential conflict of interest was reported by the author. We would like to acknowledge the assistance of volunteers in putting together this example manuscript and supplement.
% \section*{References}
\bibliographystyle{siamplain}
\bibliography{Reference}

\begin{thebibliography}{10}

\bibitem{ainsworth2017analysis}
{\sc M.~Ainsworth and Z.~Mao}, {\em Analysis and approximation of a fractional
  cahn--hilliard equation}, SIAM Journal on Numerical Analysis, 55 (2017),
  pp.~1689--1718.

\bibitem{akrivis2019energy}
{\sc G.~Akrivis, B.~Li, and D.~Li}, {\em Energy-decaying extrapolated rk--sav
  methods for the allen--cahn and cahn--hilliard equations}, SIAM Journal on
  Scientific Computing, 41 (2019), pp.~A3703--A3727.

\bibitem{ambati2015review}
{\sc M.~Ambati, T.~Gerasimov, and L.~De~Lorenzis}, {\em A review on phase-field
  models of brittle fracture and a new fast hybrid formulation}, Computational
  Mechanics, 55 (2015), pp.~383--405.

\bibitem{BerryDiffusive}
{\sc J.~Berry, M.~Grant, and K.~R. Elder}, {\em Diffusive atomistic dynamics of
  edge dislocations in two dimensions}, Physical Review E Statistical Nonlinear
  \& Soft Matter Physics, 73 (2006), p.~031609.

\bibitem{chen2019efficient}
{\sc C.~Chen and X.~Yang}, {\em Efficient numerical scheme for a dendritic
  solidification phase field model with melt convection}, Journal of
  Computational Physics, 388 (2019), pp.~41--62.

\bibitem{chen2019fast}
{\sc C.~Chen and X.~Yang}, {\em Fast, provably unconditionally energy stable,
  and second-order accurate algorithms for the anisotropic cahn--hilliard
  model}, Computer Methods in Applied Mechanics and Engineering, 351 (2019),
  pp.~35--59.

\bibitem{chen2018accurate}
{\sc L.~Chen, J.~Zhao, W.~Cao, H.~Wang, and J.~Zhang}, {\em An accurate and
  efficient algorithm for the time-fractional molecular beam epitaxy model with
  slope selection}, arXiv preprint arXiv:1803.01963,  (2018).

\bibitem{chen2018power}
{\sc L.~Chen, J.~Zhao, and H.~Wang}, {\em On power law scaling dynamics for
  time-fractional phase field models during coarsening}, arXiv preprint
  arXiv:1803.05128,  (2018).

\bibitem{cheng2021generalized}
{\sc Q.~Cheng, C.~Liu, and J.~Shen}, {\em Generalized sav approaches for
  gradient systems}, Journal of Computational and Applied Mathematics, 394
  (2021), p.~113532.

\bibitem{cheng2019highly}
{\sc Q.~Cheng, J.~Shen, and X.~Yang}, {\em Highly efficient and accurate
  numerical schemes for the epitaxial thin film growth models by using the
  {SAV} approach}, Journal of Scientific Computing, 78 (2019), pp.~1467--1487.

\bibitem{du2019maximum}
{\sc J.~L. L.~X. Du, Qiang. and Z.~Qiao}, {\em Maximum principle preserving
  exponential time differencing schemes for the nonlocal allen--cahn equation},
  SIAM Journal on Numerical Analysis, 57 (2019), pp.~875--898.

\bibitem{du2018stabilized}
{\sc Q.~Du, L.~Ju, X.~Li, and Z.~Qiao}, {\em Stabilized linear semi-implicit
  schemes for the nonlocal cahn--hilliard equation}, Journal of Computational
  Physics, 363 (2018), pp.~39--54.

\bibitem{elder2002modeling}
{\sc K.~Elder, M.~Katakowski, M.~Haataja, and M.~Grant}, {\em Modeling
  elasticity in crystal growth}, Physical review letters, 88 (2002), p.~245701.

\bibitem{ElderModeling}
{\sc K.~R. Elder, M.~Katakowski, M.~Haataja, and M.~Grant}, {\em Modeling
  elasticity in crystal growth}, Physical Review Letters, 88, p.~245701.

\bibitem{eyre1998unconditionally}
{\sc D.~J. Eyre}, {\em Unconditionally gradient stable time marching the
  cahn-hilliard equation}, MRS Online Proceedings Library Archive, 529 (1998).

\bibitem{guan2014second}
{\sc Z.~Guan, J.~S. Lowengrub, C.~Wang, and S.~M. Wise}, {\em Second order
  convex splitting schemes for periodic nonlocal cahn--hilliard and allen--cahn
  equations}, Journal of Computational Physics, 277 (2014), pp.~48--71.

\bibitem{guo2015thermodynamically}
{\sc Z.~Guo and P.~Lin}, {\em A thermodynamically consistent phase-field model
  for two-phase flows with thermocapillary effects}, Journal of Fluid
  Mechanics, 766 (2015), pp.~226--271.

\bibitem{he2007large}
{\sc Y.~He, Y.~Liu, and T.~Tang}, {\em On large time-stepping methods for the
  {C}ahn-{H}illiard equation}, Applied Numerical Mathematics, 57 (2007),
  pp.~616--628.

\bibitem{huang2020highly}
{\sc F.~Huang, J.~Shen, and Z.~Yang}, {\em A highly efficient and accurate new
  scalar auxiliary variable approach for gradient flows}, SIAM Journal on
  Scientific Computing, 42 (2020), pp.~A2514--A2536.

\bibitem{jiang2022improving}
{\sc M.~Jiang, Z.~Zhang, and J.~Zhao}, {\em Improving the accuracy and
  consistency of the scalar auxiliary variable (sav) method with relaxation},
  Journal of Computational Physics,  (2022), p.~110954.

\bibitem{li2018unconditionally}
{\sc H.~Li, L.~Ju, C.~Zhang, and Q.~Peng}, {\em Unconditionally energy stable
  linear schemes for the diffuse interface model with peng--robinson equation
  of state}, Journal of Scientific Computing, 75 (2018), pp.~993--1015.

\bibitem{li2019efficient}
{\sc Q.~Li, L.~Mei, X.~Yang, and Y.~Li}, {\em Efficient numerical schemes with
  unconditional energy stabilities for the modified phase field crystal
  equation}, Advances in Computational Mathematics, 45 (2019), pp.~1551--1580.

\bibitem{xiaoli2019energy}
{\sc S.~J. Li, Xiaoli and H.~Rui}, {\em Energy stability and convergence of sav
  block-centered finite difference method for gradient flows}, Mathematics of
  Computation, 88 (2019), pp.~2047--2068.

\bibitem{li2019sav}
{\sc X.~Li and J.~Shen}, {\em On a {SAV-MAC} scheme for the
  {C}ahn-{H}illiard-{N}avier-{S}tokes phase field model}, arXiv preprint
  arXiv:1905.08504,  (2019).

\bibitem{xiaoli2020error}
{\sc X.~Li and J.~Shen}, {\em Error analysis of the {SAV-MAC} scheme for the
  {N}avier--{S}tokes equations}, SIAM Journal on Numerical Analysis, 58 (2020),
  pp.~2465--2491.

\bibitem{li2020stability}
{\sc X.~Li and J.~Shen}, {\em Stability and error estimates of the {SAV}
  fourier-spectral method for the phase field crystal equation}, Adv Comput
  Math, 46 (2020), p.~48.

\bibitem{li2020new}
{\sc X.~Li, J.~Shen, and Z.~Liu}, {\em New {SAV}-pressure correction methods
  for the {N}avier-{S}tokes equations: stability and error analysis}, arXiv
  preprint arXiv:2002.09090,  (2020).

\bibitem{li2019energy}
{\sc X.~Li, J.~Shen, and H.~Rui}, {\em Energy stability and convergence of
  {SAV} block-centered finite difference method for gradient flows},
  Mathematics of Computation, 88 (2019), pp.~2047--2068.

\bibitem{li2017efficient}
{\sc Y.~Li and J.~Kim}, {\em An efficient and stable compact fourth-order
  finite difference scheme for the phase field crystal equation}, Computer
  Methods in Applied Mechanics and Engineering, 319 (2017), pp.~194--216.

\bibitem{lin2019numerical}
{\sc L.~Lin, Z.~Yang, and S.~Dong}, {\em Numerical approximation of
  incompressible {N}avier-{S}tokes equations based on an auxiliary energy
  variable}, Journal of Computational Physics, 388 (2019), pp.~1--22.

\bibitem{liu2019efficient}
{\sc Z.~Liu and X.~Li}, {\em Efficient modified stabilized invariant energy
  quadratization approaches for phase-field crystal equation}, Numerical
  Algorithms, 85 (2020), pp.~107--132.

\bibitem{liu2020exponential}
{\sc Z.~Liu and X.~Li}, {\em The exponential scalar auxiliary variable
  ({E-SAV}) approach for phase field models and its explicit computing}, SIAM
  Journal on Scientific Computing, 42 (2020), pp.~B630--B655.

\bibitem{liu2020two}
{\sc Z.~Liu and X.~Li}, {\em Two fast and efficient linear semi-implicit
  approaches with unconditional energy stability for nonlocal phase field
  crystal equation}, Applied Numerical Mathematics, 150 (2020), pp.~491--506.

\bibitem{marth2016margination}
{\sc W.~Marth, S.~Aland, and A.~Voigt}, {\em Margination of white blood cells:
  a computational approach by a hydrodynamic phase field model}, Journal of
  Fluid Mechanics, 790 (2016), pp.~389--406.

\bibitem{miehe2010phase}
{\sc C.~Miehe, M.~Hofacker, and F.~Welschinger}, {\em A phase field model for
  rate-independent crack propagation: Robust algorithmic implementation based
  on operator splits}, Computer Methods in Applied Mechanics and Engineering,
  199 (2010), pp.~2765--2778.

\bibitem{shen2012second}
{\sc J.~Shen, C.~Wang, X.~Wang, and S.~M. Wise}, {\em Second-order convex
  splitting schemes for gradient flows with {E}hrlich-{S}chwoebel type energy:
  application to thin film epitaxy}, SIAM Journal on Numerical Analysis, 50
  (2012), pp.~105--125.

\bibitem{shen2018scalar}
{\sc J.~Shen, J.~Xu, and J.~Yang}, {\em The scalar auxiliary variable ({SAV})
  approach for gradient flows}, Journal of Computational Physics, 353 (2018),
  pp.~407--416.

\bibitem{ShenA}
{\sc J.~Shen, J.~Xu, and J.~Yang}, {\em A new class of efficient and robust
  energy stable schemes for gradient flows}, SIAM Review, 61 (2019),
  pp.~474--506.

\bibitem{shen2010numerical}
{\sc J.~Shen and X.~Yang}, {\em Numerical approximations of {A}llen-{C}ahn and
  {C}ahn-{H}illiard equations}, Discrete Contin. Dyn. Syst, 28 (2010),
  pp.~1669--1691.

\bibitem{shen2015efficient}
{\sc J.~Shen, X.~Yang, and H.~Yu}, {\em Efficient energy stable numerical
  schemes for a phase field moving contact line model}, Journal of
  Computational Physics, 284 (2015), pp.~617--630.

\bibitem{shin2016first}
{\sc J.~Shin, H.~G. Lee, and J.-Y. Lee}, {\em First and second order numerical
  methods based on a new convex splitting for phase-field crystal equation},
  Journal of Computational Physics, 327 (2016), pp.~519--542.

\bibitem{StefanovicPhase}
{\sc P.~N. Stefanovic~P, Haataja~M}, {\em Phase field crystal study of
  deformation and plasticity in nanocrystalline materials}, Physical Review E,
  80 (2009), p.~046107.

\bibitem{WangEfficient}
{\sc X.~Wang, L.~Ju, and Q.~Du}, {\em Efficient and stable exponential time
  differencing runge¨ckutta methods for phase field elastic bending energy
  models}, Journal of Computational Physics, 316, pp.~21--38.

\bibitem{weng2017fourier}
{\sc Z.~Weng, S.~Zhai, and X.~Feng}, {\em A fourier spectral method for
  fractional-in-space cahn--hilliard equation}, Applied Mathematical Modelling,
  42 (2017), pp.~462--477.

\bibitem{wheeler1993computation}
{\sc A.~A. Wheeler, B.~T. Murray, and R.~J. Schaefer}, {\em Computation of
  dendrites using a phase field model}, Physica D: Nonlinear Phenomena, 66
  (1993), pp.~243--262.

\bibitem{WuPhase}
{\sc K.-A. Wu, A.~Adland, and A.~Karma}, {\em Phase-field-crystal model for fcc
  ordering}, Physical Review E Statistical Nonlinear \& Soft Matter Physics, 81
  (2010), p.~061601.

\bibitem{yang2016linear}
{\sc X.~Yang}, {\em Linear, first and second-order, unconditionally energy
  stable numerical schemes for the phase field model of homopolymer blends},
  Journal of Computational Physics, 327 (2016), pp.~294--316.

\bibitem{yang2018numerical}
{\sc X.~Yang}, {\em Numerical approximations for the {C}ahn--{H}illiard phase
  field model of the binary fluid-surfactant system}, Journal of Scientific
  Computing, 74 (2018), pp.~1533--1553.

\bibitem{yang2017linearly}
{\sc X.~Yang and D.~Han}, {\em Linearly first-and second-order, unconditionally
  energy stable schemes for the phase field crystal model}, Journal of
  Computational Physics, 330 (2017), pp.~1116--1134.

\bibitem{yang2017numerical}
{\sc X.~Yang and G.~Zhang}, {\em Numerical approximations of the
  {C}ahn-{H}illiard and {A}llen-{C}ahn equations with general nonlinear
  potential using the {I}nvariant {E}nergy {Q}uadratization approach}, arXiv
  preprint arXiv:1712.02760,  (2017).

\bibitem{yang2020convergence}
{\sc X.~Yang and G.-D. Zhang}, {\em Convergence analysis for the invariant
  energy quadratization ({IEQ}) schemes for solving the cahn--hilliard and
  allen--cahn equations with general nonlinear potential}, Journal of
  Scientific Computing, 82 (2020), pp.~1--28.

\bibitem{YangNumerical}
{\sc X.~Yang, J.~Zhao, and Q.~Wang}, {\em Numerical approximations for the
  molecular beam epitaxial growth model based on the invariant energy
  quadratization method}, Journal of Computational Physics, 333 (2017),
  pp.~104--127.

\bibitem{yang2020roadmap}
{\sc Z.~Yang and S.~Dong}, {\em A roadmap for discretely energy-stable schemes
  for dissipative systems based on a generalized auxiliary variable with
  guaranteed positivity}, Journal of Computational Physics, 404 (2020),
  p.~109121.

\bibitem{zhai2014numerical}
{\sc S.~Zhai, X.~Feng, and Y.~He}, {\em Numerical simulation of the three
  dimensional allen--cahn equation by the high-order compact adi method},
  Computer Physics Communications, 185 (2014), pp.~2449--2455.

\bibitem{zhu1999coarsening}
{\sc J.~Zhu, L.-Q. Chen, J.~Shen, and V.~Tikare}, {\em Coarsening kinetics from
  a variable-mobility {C}ahn-{H}illiard equation: {A}pplication of a
  semi-implicit {F}ourier spectral method}, Physical Review E, 60 (1999),
  p.~3564.

\end{thebibliography}

\end{document}